\documentclass[11pt]{amsart}
\usepackage{amsmath,amssymb,latexsym,esint,cite,mathrsfs}
\usepackage{verbatim,wasysym}
\usepackage[left=2.5cm,right=2.5cm,top=2.7cm,bottom=2.7cm]{geometry}
\usepackage{tikz,enumitem,graphicx, subfig, microtype, color}
\usepackage{epic,eepic}

\usepackage[colorlinks=true,urlcolor=blue, citecolor=red,linkcolor=blue,
linktocpage,pdfpagelabels, bookmarksnumbered,bookmarksopen]{hyperref}
\usepackage[hyperpageref]{backref}
\usepackage[english]{babel}
\allowdisplaybreaks
\numberwithin{equation}{section}
\newtheorem{theorem}{Theorem}[section]
\newtheorem{proposition}[theorem]{Proposition}
\newtheorem{corollary}[theorem]{Corollary}
\newtheorem{lemma}[theorem]{Lemma}
\theoremstyle{definition}

\newtheorem{definition}[theorem]{Definition}
\newtheorem{remark}[theorem]{Remark}

\newcommand{\R}{\mathbb{R}}

\newcommand{\supp}{{\rm supp}{\hspace{.05cm}}}
\begin{document}
\title
 [Normalized solutions for $p$-Laplacian equations with log nonlinearity]
 {Existence and concentration of normalized solutions for $p$-Laplacian equations with logarithmic nonlinearity}

\author[L.J. Shen]{Liejun Shen}

\address{Liejun Shen, \newline\indent
Department of Mathematics, \newline\indent
Zhejiang Normal University, \newline\indent
 Jinhua, Zhejiang, 321004, People's Republic of China}
\email{\href{mailto:ljshen@zjnu.edu.cn.}{ljshen@zjnu.edu.cn.}}

 \author[M.\ Squassina]{Marco Squassina}

 \address{Marco Squassina, \newline\indent
 	Dipartimento di Matematica e Fisica \newline\indent
 	Universit\`a Cattolica del Sacro Cuore, \newline\indent
 	Via della Garzetta 48, 25133, Brescia, Italy}
 \email{\href{mailto:marco.squassina@unicatt.it.}{marco.squassina@unicatt.it.}}


\thanks{L.J. Shen is partially supported by NSFC (12201565).
 M.\  Squassina  is  member  of  Gruppo  Nazionale  per
 	l'Analisi  Matematica,  la Probabilita  e  le  loro  Applicazioni  (GNAMPA)  of  the  Istituto  Nazionale  di  Alta  Matematica  (INdAM)}

\begin{abstract}
We investigate the existence and concentration of normalized solutions for a $p$-Laplacian
problem with logarithmic nonlinearity of type
\[
 \left\{
   \begin{array}{ll}
   \displaystyle -\varepsilon^p\Delta_p u+V(x)|u|^{p-2}u=\lambda |u|^{p-2}u+|u|^{p-2}u\log|u|^p
~\text{in}~\R^N, \\
    \displaystyle   \int_{\R^N}|u|^pdx=a^p\varepsilon^N,
   \end{array}
 \right.
\]
where $a,\varepsilon> 0$, $\lambda\in\R$ is known as the Lagrange multiplier, $\Delta_p\cdot =\text{div} (|\nabla \cdot|^{p-2}\nabla \cdot)$
 denotes the usual
$p$-Laplacian operator with $2\leq p < N$ and $V \in \mathcal{C}^0(\R^N)$ is the potential which satisfies some suitable
assumptions. We prove that the number of positive solutions depends on the profile of $V$ and
each solution concentrates around its corresponding global minimum point of $V$ in the semiclassical
limit when $\varepsilon\to0^+$ using variational method.
 Moreover, we also get the existence of normalized solutions
for some logarithmic $p$-Laplacian equations involving mass-supercritical nonlinearities.
\end{abstract}
\maketitle

\begin{center}
	\begin{minipage}{8.5cm}
		\small
		\tableofcontents
	\end{minipage}
\end{center}
\medskip

\section{Introduction and main results}

In this article, we aim to establish the existence
and concentrating behavior of nontrivial solutions for the following $p$-Laplacian equations
with logarithmic nonlinearity of type
\begin{equation}\label{mainequation1}
-\varepsilon^p\Delta_p u+V(x)|u|^{p-2}u=\lambda |u|^{p-2}u+|u|^{p-2}u\log|u|^p
~\text{in}~\R^N,
\end{equation}
under the constraint
\begin{equation}\label{mainequation1a}
\int_{\R^N}|u|^pdx=a^p\varepsilon^N,
\end{equation}
where $a,\varepsilon> 0$, $\lambda\in\R$ is known as the Lagrange multiplier, $\Delta_p\cdot=\text{div} (|\nabla \cdot|^{p-2}\nabla \cdot)$
 denotes the usual
$p$-Laplacian operator with $2\leq  p < N$ and $V \in \mathcal{C}^0(\R^N)$ is the potential
which satisfies
\begin{itemize}
  \item[$(V_1)$] $V\in \mathcal{C}^0(\mathbb{R}^N)$  and $-\infty<V_0=
  \inf\limits_{x\in\mathbb{R}^N}V(x)<V_\infty=\lim\limits_{|x|\to+\infty}V(x)<+\infty$;
  \item[$(V_2)$] $V^{-1}(\{V_{0}\})=\{x^1,x^2,\ldots,x^l\}$ with $x^1=0$ and $x^i\ne x^j$ if $i\ne j$
  for all $i,j\in\{1,2,\cdots,l\}$.
\end{itemize}

In the case $p=2$, Eq. \eqref{mainequation1} also comes from the study of solitary waves for the following nonlinear
Schr\"{o}dinger equation
\begin{equation}\label{Introduction1}
\text{i}\dfrac{\partial\psi}{\partial t}+\Delta\psi-V(x)\psi+h(|\psi|^2)\psi=0~\text{in}~(0,\infty)\times\mathbb{R}^N,
\end{equation}
where i is the imaginary unit, $\psi:(0,\infty)\times\mathbb{R}^{N}\rightarrow\mathbb{C}$,
$V:\R^N\to\R$ is the potential,
$h(\mathrm{e}^\mathrm{i\theta}z)=\mathrm{e}^{\mathrm{i}\theta}h(z)$
for $z\in\mathbb{C}$ and $\theta\in\mathbb{R}$.
It is not difficult to see that any solution $\psi$ of Eq. \eqref{Introduction1} with the Cauchy initial
function $\psi(0, x)$ preserves the $L^2$-mass, namely
$$
\int_{\mathbb{R}^N}|\psi(t,x)|^2\mathrm{d}x=\int_{\mathbb{R}^N}|\psi(0,x)|^2\mathrm{d}x,~\forall t\in(0,\infty).
$$
Actually, Eq. \eqref{Introduction1} is usually adopted in the study of nonlinear optics and Bose-Einstein condensates,
 where $\psi$ describes the state of the condensate and the $L^2$-mass is the total number of atoms,
 see e.g. \cite{Fibich,TVZ,Zhang}. One significant motivation associated with Eq. \eqref{Introduction1} is the search for its standing wave solutions.
The standing wave is a solution of the form $\psi(t,x) =e^{-\text{i}\lambda t}u(x)$,
where $\lambda\in \mathbb{R}$ and $u: \mathbb{R} ^N\to \mathbb{R}$ is a time-independent function that satisfies
the nonlinear elliptic equation
\begin{equation}\label{Introduction2}
- \Delta u+ V( x) u= \lambda u+ g(u)~\text{in}~ \mathbb{R}^N,
\end{equation}
with $g(u)=h(|u|^2)u$.

Usually, there are two directions
 to the study of standing waves of the Schr\"{o}dinger equation
\eqref{Introduction2}. On the one hand, one can choose the frequency $\lambda$ to be fixed and look for the
existence of nontrivial solutions for it by investigating critical points of the variational functional
$J_\lambda:H^1(\R^N)\to\R$ defined by
$$
\mathcal{J}_\lambda(u)=\frac{1}{2}\int_{\mathbb{R}^N}\left(|\nabla u|^2+(V(x)-\lambda)|u|^2\right)dx-\int_{\mathbb{R}^N}G(u)dx,
$$
where $G(t)=\int_0^tg(s)ds$.
When $g(t)=t\log t^2$, we refer to the article \cite{Zloshchastiev}
and its references therein to acquaint the
 significant physical
applications in quantum mechanics, quantum optics, nuclear physics, transport and diffusion
phenomena, open quantum systems, effective quantum gravity, superfluidity theory
and Bose-Einstein condensation for Eq. \eqref{Introduction2}.
Owing to the logarithmic type nonlinearity,
it presents some challenging mathematical problems. For instance, the
associated variational functional is not $\mathcal{C}^1$-smooth
since one can find a function below
$$
\left.u(x)=\left\{\begin{array}{ll}(|x|^{N/2}\log(|x|))^{-1},&\quad|x|\geq3,\\0,&\quad|x|\leq2,\end{array}\right.\right.
$$
such that $u\in H^1(\R^N)$, however it holds that $\int_{\R^N}u^2\log u^2dx=-\infty$.

The initial work to deal with this difficulty is due to Cazenave in \cite{Cazenave},
where the author considered the following
 logarithmic Schr\"{o}dinger equation
$$
\text{i}u_t+\Delta u
 +u\log u^2=0,~(t,x)\in\mathbb{R}\times\mathbb{R}^N,
$$
in the space $W\triangleq\{u\in H^1(\R^N):\int_{\R^N}|u^2\log u^2|dx<\infty\}$
with a suitable Luxemburg norm. Speaking it more clearly, by introducing the $N$-function
\begin{equation}\label{AA}
 \left.A(s)=\left\{\begin{array}{ll} -s^2\log s^2,&\quad0\leq s\leq e^{-3},\\
  3s^2+4e^{-3}s-e^{-6},
&\quad s\geq e^{-3},\end{array}\right.\right.
\end{equation}
$$
\|u\|_A=\inf\left\{\lambda>0:\int_{\R^N}A\left(\frac{|u|}{\lambda}\right)dx\leq1\right\},
$$
and $(\|\cdot\|_{H^1(\R^N)}+\|\cdot\|_A)$ as the Luxemburg norm,
then the author obtained the existence of infinitely many critical
points for the variational functional
$$
\mathcal{J} (u)=\frac{1}{2}\int_{\mathbb{R}^N} |\nabla u|^2 dx-\int_{\mathbb{R}^N}u^2\log u^2dx, ~u\in W,
$$
on the set $\Sigma\triangleq\{u\in W:\int_{\mathbb{R}^N}u^2dx=1\}$.
Very recently, Alves and his collaborators  \cite{Alves1,Alves2,Alves3}
have used the decomposition
\begin{equation}\label{AJdecomposition}
\hat{F}_2(s)-\hat{F}_1(s)=\frac{1}{2}s^2\log s^2,~ s\in\mathbb{R},
\end{equation}
where,
$$
\left.\hat{F}_1 (s)=\left
\{\begin{array}{lr}
0, & s\leq0,\\
-\frac{1}{2}s^2\log s^2, &0<s< \delta,\\
 -\frac{1}{2}s^2\big[\log \delta ^2+3\big]+2\delta s -\frac{1}{2} \delta ^2,&s\geq \delta,
\end{array}\right.\right.
 $$
and
$$
\left.\hat{F}_2(s)=\left\{\begin{array}{ll}0,&|s|\le \delta,\\
\frac{1}{2}|s|^2\log\big(|s|^2/ \delta ^2\big)+2\delta|s| -\frac{3}{2}|s|^2-\frac{1}{2} \delta ^2,&|s|\ge \delta,
\end{array}\right.\right.
$$
for some sufficiently small $\delta>0$, and introduced the Orlicz space
$$
L^{\hat{F}_1}(\R^N)=\left\{u\in L^1_{\text{loc}}(\R^N):\int_{\R^N}\hat{F}_1
\left(\frac{|u|}{\lambda}\right)dx<+\infty ~\text{for some}~\lambda>0\right\},
$$
then studied some different types of
logarithmic Schr\"{o}dinger equation
in the space $H^1(\R^N)\cap L^{\tilde{F}_1}(\R^N)$.
Of course, there are various techniques in the literature,
see e.g. \cite{AJ1,AJ2,AJ3,d'Avenia,zenari,Ferriere,Ikoma,Shuai,Squassina,Mederski,WangZhang} and the references therein even if
these ones are far to be exhaustive,
to find some other interesting works on logarithmic Schr\"{o}dinger equations.

On the other hand, one can deal with the case $\lambda\in\R$ is unknown. In this situation, $\lambda\in\R$
appears as a Lagrange multiplier and the $L^2$-norm of solution is prescribed. From the physical
point of view,
the research holds particular significance as it accounts for the
conservation of mass. Additionally, it provides valuable insights into the dynamic properties
of standing waves in Eq. \eqref{Introduction2}, such as stability or instability \cite{BC,CP}.
In \cite{Jeanjean1997}, combining a minimax approach and compactness argument,
  Jenajean contemplated the existence of solutions for the following Schr\"{o}dinger
 problem
 \begin{equation}\label{Jeanjean}
 \left\{
   \begin{array}{ll}
 \displaystyle   -\Delta u+\lambda u=
 g(u) ~\text{in}~ \R^N, \\
 \displaystyle     \int_{\R^N}|u|^2dx=a^2>0.
   \end{array}
 \right.
\end{equation}
 Later on, there are some complements and generalizations in \cite{JeanjeanLu2}.
 In \cite{Soave1}, for $g(t)=\mu|t|^{q-2}t+|t|^{p-2}t$ with $2<q\leq 2+\frac4N\leq p<2^*$,
  Soave considered the existence of solutions for problem \eqref{Jeanjean},
  where $2^*=\frac{2N}{N-2}$ if $N\geq3$ and $2^*=\infty$ if $N=2$.
  For this type of combined nonlinearities, Soave \cite{Soave2} proved
the existence of ground state and excited solutions when $p=2^*$.
For more interesting results for problem \eqref{Jeanjean},
we refer the reader to \cite{BartschSoave,JeanjeanLu1,JeanjeanLu3,LiXinfu,WeiWu} and the references therein.

In reality, the $p$-Laplacian operator in Eq. \eqref{mainequation1} appears in many nonlinear problems,
see \cite{Diaz,Gilbarg,Mastorakis} for example,
in fluid dynamics, the shear stress $\vec{\tau}$ combined with the velocity gradient
$\nabla u$ of the fluid are corresponding to the manner that
$$
 \vec{\tau}=r|\nabla u|^{p-2}\nabla u.
 $$
The fluid is dilatant, pseudoplastic or Newtonian when $p > 2$, $p < 2$ or $p = 2$, respectively.
Therefore, the equation governing the motion of the fluid includes the $p$-Laplacian.
Moreover, such an
operator also appears in the study of flows through porous media (when $p = 3/2$),
nonlinear elasticity (when $p \geq2$) and glaciology (when $p\in(1, 4/3]$).

In light of the physical background of $p$-Laplacian operator,
motivated by the prescribed $L^2$-norm solution
for Eq. \eqref{Introduction2},
 it is interesting to investigate the
$p$-Laplacian equation under the $L^p$-mass constraint
and regard the frequency $\lambda\in\R$ as a Lagrange multiplier.
Up to the best knowledge of us,
 there exist very few articles on this topic.
 Wang \emph{et al.} \cite{WLZL} established the existence of solutions
 for the problem
$$
\left\{
   \begin{array}{ll}
 \displaystyle -\Delta_pu+|u|^{p-2}u=\lambda u+|u|^{s-1}u~\text{in}~\mathbb{R}^N, \\
 \displaystyle     \int_{\R^N}|u|^2dx=a^2>0.
   \end{array}
 \right.
$$
when $a>0$ is sufficiently small,
where $\max\left\{1,\frac{2N}{N+2}\right\}<p<N$ and $s\in\left(\frac{N+2}Np,p^*\right)$
with $p^*=\frac{Np}{N-p}$.
In \cite{ZZ}, the authors studied the $p$-Laplacian equation with a $L^p$-norm constraint
$$
\begin{cases}-\Delta_pu=\lambda|u|^{p-2}u+\mu|u|^{q-2}u+g(u)~\text{in}~\mathbb{R}^N,\\
\displaystyle\int_{\mathbb{R}^N}|u|^p\mathrm{d}x=a^p,\end{cases}
$$
where $1<p<q\leq \bar{p}\triangleq p+\frac{p^2}N$
and $g\in \mathcal{C}^0(\R,\R)$ is odd and $L^p$-supercritical.
For the suitable $\mu$, they obtained several existence results
including the existence of infinitely many solutions
by Schwarz rearrangement technique, Ekeland variational principle
and the Fountain theorem. There exist some other similar results
for the $p$-Laplacian equation prescribed a $L^p$-norm, see e.g.
\cite{DW1,DW2,ZLL}.

Whereas, as far as we are concerned, there seems no related results
for the $p$-Laplacian equation with a logarithmic nonlinearity under the $L^p$-mass constraint,
and so one of the aims in the present article is to fulfill the blank.
Motivated by \cite{Alves,Alves1,AT}, we shall derive the multiplicity
and concentrating behavior of positive solutions for a
logarithmic $p$-Laplacian equation \eqref{mainequation1}
under the constraint \eqref{mainequation1a}.

Now, we can state the first main result as follows.

\begin{theorem}\label{maintheorem1} Let $2\leq p<N$
and $(V_1)-(V_2)$. Then, there exists a $\varepsilon^*>0$ such that
\eqref{mainequation1}-\eqref{mainequation1a} possesses at least $l$ different couples of weak
solutions $(u^j_\varepsilon,\lambda^j_\varepsilon)\in W^{1,p}(\R^N)\times\R$
for all $\varepsilon\in(0,\varepsilon^*)$ with $u^j_\varepsilon(x)>0$ for every $x\in\R^N$ and
 $\lambda^j_\varepsilon<0$, where $j\in\{1,2,\cdots,l\}$.
Moreover, each $u_j^\varepsilon$ has a maximum point $z^j_\varepsilon\in\R^N$ such that $V(z^j_\varepsilon)\to V(x^j)=V_0$
as $\varepsilon\to0^+$. Besides, there exist two constants $C_0^j>0$ and $c_0^j>0$ satisfying
$$
u^j_\varepsilon\leq  C_0^j\exp \bigg(-c_0^j\frac {|x-z^j_\varepsilon|}{\varepsilon}\bigg)
$$
for all $\varepsilon\in(0,\varepsilon^*)$ and $x\in\R^N$.
\end{theorem}

\begin{remark}
A recent work by Alves and Ji we would like to mention here is the paper \cite{Alves1}, where the authors
contemplated the following Schr\"{o}dinger problem with logarithmic nonlinearity
\begin{equation}\label{AJequation}
  \left\{
   \begin{array}{ll}
   \displaystyle -\varepsilon^2\Delta  u+V(x) u=\lambda  u+ u\log|u|^2
~\text{in}~\R^N, \\
    \displaystyle     \int_{\R^N}|u|^2dx=a^2\varepsilon^N.
   \end{array}
 \right.
\end{equation}
Here the potential $V$ satisfies the assumptions
\begin{itemize}
  \item[$(\hat{V}_1)$] $V\in \mathcal{C}^0(\mathbb{R}^N)$  and $-1\leq V_0=
  \inf\limits_{x\in\mathbb{R}^N}V(x)<V_\infty=\lim\limits_{|x|\to+\infty}V(x)<+\infty$;
  \item[$(\hat{V}_2)$] $M=\{x\in\R^N:V(x)=V_0\}$ and $M_\delta=\{x\in\R^N:\text{dist}(x,M)\leq\hat{\delta}\}$ for some $\hat{\delta}>0$.
\end{itemize}
They derived that, given a $\hat{\delta}>0$, there exist $\hat{a}>0$ and $\hat{\varepsilon}>0$
such that problem \eqref{AJequation} possesses at lease $\text{cat}_{M_{\hat{\delta}}}(M)$
couples of solutions for all $a>\hat{a}$ and $\varepsilon\in(0,\hat{\varepsilon})$, where the concentration has also been studied
but the property of exponential decay of solutions is absent.
In contrast to it, there are three main contributions in the present paper which are exhibited below

(I) We just suppose that $V_0>-\infty$ in $(V_1)$ instead of $V_0\geq-1$ in $(\hat{V}_1)$;

(II) There is no restriction on the mass $a>0$ in Theorem \ref{maintheorem1}, while the results in \cite{Alves1} hold true
provided that the mass $a>0$ is sufficiently large;

(III) The $p$-Laplacian operator appearing in Eq. \eqref{mainequation1} is non-homogenous which reveals that
the calculations would be more complicated and technical. For instance,
as a counterpart of \eqref{AJdecomposition}, it is nontrivial to
  construct some functions $F_1$ and $F_2$ such that
\begin{equation}\label{decomposition}
  F_2(s)-F_1(s)=\frac1p|s|^p\log|s|^p,~\forall s\in\R.
\end{equation}
 Moreover, we shall make some efforts to investigate the property of exponential decay of obtained solutions
 which should be regarded as a replenishment. As a consequence, we could never repeat the approaches adopted in \cite{Alves1}
 to conclude the proof of Theorem \ref{maintheorem1}.
\end{remark}

We note that the considerations of Eq. \eqref{mainequation1} can date back to
the studies of semiclassical problems for Schr\"{o}dinger equations,
the reader could refer to \cite{AM,BF} for some detailed survey on such topic which comes from the
pioneering research work by Floer and Weinstein \cite{FW}. Soon afterwards, this topic on different types of
Schr\"{o}dinger equations
has been investigated extensively under several distinct hypotheses on the potential and the nonlinearity,
see e.g. \cite{ABC,DF1,DF2,Ding1,Ding2,JT,Rabinowitz,Wang} and the references therein.
Hence, it permits us to follow the effective procedures in the literature to handle the $p$-Laplacian problems with $2\leq p<N$
and logarithmic nonlinearity
in this paper.

Performing the scaling $v(x)=u(\varepsilon x)$,
one could observe that, to consider \eqref{mainequation1}-\eqref{mainequation1a}, it is equivalent to study the problem
\begin{equation}\label{mainequation2}
 \left\{
   \begin{array}{ll}
   \displaystyle -  \Delta_p v+V(\varepsilon x)|v|^{p-2}v=\lambda |v|^{p-2}v+|v|^{p-2}v\log|v|^p
~\text{in}~\R^N, \\
    \displaystyle     \int_{\R^N}|v|^pdx=a^p.
   \end{array}
 \right.
\end{equation}
In other words, if the couple $(v,\lambda)$ is a (weak)
solution of Problem \eqref{mainequation2}, then
$(u,\lambda)$ is a
solution of \eqref{mainequation1}-\eqref{mainequation1a},
where $v(x)=u(\varepsilon x)$ for all $x\in\R^N$.
Let $v=\sigma w$ with some $\sigma > 0$, we shall observe that the couple $(v,\lambda)$ is a weak solution to Problem
\eqref{mainequation2} if and only if $(w,\lambda)$ is a weak solution to problem below
\begin{equation}\label{mainequation3}
 \left\{
   \begin{array}{ll}
   \displaystyle -  \Delta_p w+[V(\varepsilon x)-\log \sigma^{p}]|w|^{p-2}w=\lambda |w|^{p-2}w+|w|^{p-2}w\log|w|^p
~\text{in}~\R^N, \\
    \displaystyle     \int_{\R^N}|w|^pdx=a^p\sigma^{-p}.
   \end{array}
 \right.
\end{equation}
At this stage, we know that if one wants to study problems \eqref{mainequation1}-\eqref{mainequation1a},
it would be enough to deal with Problem \eqref{mainequation3}.
Since $\sigma>0$ is arbitrary,
we are able to find a sufficiently small $\sigma>0$ such that
$$
V(\varepsilon x)-\log \sigma^{p}\geq -1,~\forall x\in\R^N,
$$
since $V_0=\inf\limits_{x\in\R^N}V(x)>-\infty$ by $(V_1)$ and
$ a^p\sigma^{-p}>\tilde{a}$,
where $\tilde{a}>0$ is a fixed constant.

Owing to the above discussions, to deduce Theorem \ref{maintheorem1},
we just need to prove the following result.

\begin{theorem}\label{maintheorem2} Let $2\leq p<N$
and $(\hat{V}_1)-(V_2)$. Then, there exist $a^*>0$ and $\varepsilon^*>0$ such that
\eqref{mainequation1}-\eqref{mainequation1a} has at least $l$ different couples of weak solutions
 $(u^j_\varepsilon,\lambda^j_\varepsilon)\in W^{1,p}(\R^N)\times\R$
for all $a>a^*$ and $\varepsilon\in(0,\varepsilon^*)$ with $u^j_\varepsilon(x)>0$ for every $x\in\R^N$ and
 $\lambda^j_\varepsilon<0$, where $j\in\{1,2,\cdots,l\}$.
Moreover, each $u^j_\varepsilon$ admits a maximum point $z^j_\varepsilon\in\R^N$ such that $V(z^j_\varepsilon)\to V(x^j)=V_0$
as $\varepsilon\to0^+$. Besides, there exist two constants $C_0^j>0$ and $c_0^j>0$ satisfying
$$
u^j_\varepsilon\leq  C_0^j\exp  \bigg(-c_0^j \frac{|x-z^j_\varepsilon|}{\varepsilon}  \bigg )
$$
for all $\varepsilon\in(0,\varepsilon^*)$ and $x\in\R^N$.
\end{theorem}

Now, we shall turn to investigate the existence of normalized solutions
for logarithmic $p$-Laplacian equations with mass subcritical and supercritical nonlinearities.
First of all, we recall the Gagliardo-Nirenberg
inequality (see e.g. \cite{Agueh,Weinstein}), for every $p<s<p^*$, there exists an optimal constant $\mathbb{C}_{N,p,s}>0$
 depending only
on $N$, $p$ and $s$ such that
\begin{equation}\label{GN}
\|u\|_{L^s(\R^N)}\leqslant \mathbb{C}_{N,p,s}
\|\nabla u\|_{L^p(\R^N)}^{\beta_s}\|u\|_{L^p(\R^N)}^{1-\beta_s},~\forall u\in W^{1,p}(\mathbb{R}^N),
\end{equation}
 where
 \begin{equation}\label{GN2}
 \beta_s\triangleq N\left(\frac1p-\frac1s\right)=   \frac{N(s-p)}{ps}.
\end{equation}
Due to \eqref{GN}, one observes that
$$
\bar{p}=p+\frac{p^2}N
$$
is the $L^p$-critical exponent
with respect to $p$-Laplacian equation.
Indeed, we focus on establishing the existence and multiplicity of positive solutions
for the following $p$-Laplacian type problems
\begin{equation}\label{mainequation4}
 \left\{
   \begin{array}{ll}
   \displaystyle -  \Delta_p u =\lambda |u|^{p-2}u+|u|^{p-2}u\log|u|^p+\mu|u|^{q-2}u
~\text{in}~\R^N, \\
    \displaystyle     \int_{\R^N}|u|^pdx=a^p,
   \end{array}
 \right.
\end{equation}
where $a>0$, $\mu>0$ and $q\in(p,\bar{p})\cup(\bar{p},p^*)$.

In order to contemplate the Problem \eqref{mainequation4} involving a class of pure-power type mass subcritical and supercritical nonlinearities,
we shall continue to establish the existence of global minimizer of a minimization problem, namely
$$
m(a)=\inf_{u\in S(a)}J(u),
$$
where the variational functional $J:X\to\R$ is given by
\begin{equation}\label{JJ}
  J(u) =\frac{1}{p}\int_{\mathbb{R}^{N}}\left(|\nabla u |^{p}+|u |^{p}\right)dx
-\frac{1}{p}\int_{\mathbb{R}^{N}}|u |^{p}\log|u |^{p}dx
-\frac{\mu}{q}\int_{\mathbb{R}^{N}}|u |^{q} dx
\end{equation}
and the constrained set
$$
S(a)=\left\{u\in X:\int_{\R^N}|u|^pdx=a^p\right\}.
$$
Here $a>0$ and the space $X$ can be found in Section \ref{Preliminaryresults} below.

The main results for the Problem \eqref{mainequation4} involving mass-subcritical nonlinearities
can be stated as follows.

\begin{theorem}\label{maintheorem3}
Suppose that $2\leq p<N$, $\mu>0$ and $p<q< \bar{p}$. Then, there is an $a_*>0$ such that, for all fixed $a>a_*$,
$m(a)$ admits a minimizer $u\in S(a)$ which is positive and radially symmetric
and decreasing in $r=|x|$. Moreover, there is a $\lambda\in\R$ such that $(u,\lambda)\in X\times\R$ solves Problem \eqref{mainequation4}.
\end{theorem}

\begin{remark}
 It should be stressed here that Theorem \ref{maintheorem3} partially
 generalizes  \cite[Theorem 1.1]{SY}. Nevertheless, one could never simply repeat the arguments explored in it
 to arrive at our result since it needs some efforts to construct a suitable $N$-function
 like \eqref{AA} in the $p$-Laplacian setting.
 In order to avoid this obstacle, we continue to make use of the decomposition in \eqref{decomposition}
 which should be regard as one of novelties in this article.
 Whereas, the biggest challenge in the proof of Theorem \ref{maintheorem3}
 is the lack of compactness. Explaining it more clearly, the imbedding $W_r\hookrightarrow L^2(\R^N)$
 in \cite{SY} is compact, see e.g. \cite[Proposition 3.1]{Cazenave}, but we indeed can not conclude the compact
 imbedding $X_r\hookrightarrow L^p(\R^N)$ in advance.
 To overcome this difficulty, we shall introduce some new analytic tricks to recover the desired
 compactness, see Lemma \ref{attained} for instance.
 Moreover, the reader would observe that the case $\mu=0$ in Theorem \ref{maintheorem3}
acts as a special one of Theorem \ref{limittheorem} below.
 Alternatively, we successfully put forward a totally different argument to look for a minimizer of $m(a)$.
 \end{remark}

Finally, we begin paying attention to the case $p+\frac{p^2}{N}<q<p^*$
which leads to the Problem \eqref{mainequation4} a mass-supercritical one
and prove the following result.

\begin{theorem}\label{maintheorem4}
Suppose that $2\leq p<N$ and $p+\frac{p^2}{N}<q<p^*$,
there exist $\underline{a}^*>0$ and $\underline{\mu}^*>0$ such that,
  for all $a>\underline{a}^*$ and $\mu\in (0,\underline{\mu}^*)$,
  Problem \eqref{mainequation4} has a  couple of weak solution
 $(\underline{u}^*,\underline{\lambda}^*)\in X\times\R$ with $\underline{u}^*(x)>0$ for all $x\in\R^N$.
\end{theorem}

Given a $u\in S(a)$, then
$u_t(\cdot)\triangleq t^{\frac{N}{p}}u(t\cdot)$ for all $t>0$
and hence, it is elementary to arrive at the calculations
$$
J(u_t) =\frac{t^p}{p}\int_{\mathbb{R}^{N}} |\nabla u |^{p} dx+\frac{a^p(1-N\log t)}{p}
-\frac{1}{p}\int_{\mathbb{R}^{N}}|u |^{p}\log|u |^{p}dx
-\frac{\mu}{q}t^{N\left(\frac qp-1\right)}\int_{\mathbb{R}^{N}}|u |^{q} dx\to-\infty
$$
as $t\to+\infty$ which indicates that
$m(a)=-\infty$ for all $a>0$.
To get around this obstacle,
in a similar spirit of
\cite{BartschSoave},
one may depend on the following Poho\u{z}aev manifold
$$
\mathcal{P}(a)=\{u\in S(a):P(u)=0\},
$$
where the variational functional $P:S(a)\to\R$ is defined by
$$
P(u)=\int_{\mathbb{R}^{N}} |\nabla u |^{p} dx-\frac{ N }{p}a^p
- \mu N\left(\frac 1p-\frac{1}{q}\right)\int_{\mathbb{R}^{N}}|u |^{q} dx.
$$
According to the discussions in Section \ref{Preliminaryresults} below,
we can conclude that $\mathcal{P}$ is a natural constraint
since $P(u)\equiv0$ provided that $u\in X$ is a nontrivial weak solution of
the first equation in Problem \eqref{mainequation4}.
Unfortunately, we fail to argue as \cite{SY} to consider the following minimization problem
$$
m_p(a)=\inf_{u\in \mathcal{P}(a)}J(u)
$$
since it seems impossible to get the compactness of its corresponding minimizing sequence.
Actually, we even cannot make sure that $m_p(a)\leq0$. Thereby, it differs evidently from the counterparts in
\cite{SY}. Besides, we also try to study the following problem
$$
m_p^R(a)=\inf_{u\in S(a)\cap\{\|\nabla u\|_{L^p(\R^N)}<\mathcal{R}\}}J(u)
$$
combined with the Ekeland variational principle.
Although one can deduce that $m_p^\mathcal{R}(a)<0$ for some suitable $\mathcal{R}>0$,
the absence of the monotone property with respect to $m_p(a)$ results in
the lack of compactness and it is still hard to show that
$m_p^\mathcal{R}(a)$ can be attained.

In consideration of the explanations exhibited above,
motivated by the ideas introduced in \cite{AS},
we shall rely heavily on the so-called truncation argument.
Let us introduce it step by step.
For every $\mathcal{R}>0$ and $p<\bar{q}< p+\frac{p^2}N$, we define
the auxiliary function $f_{\mathcal{R}}:\mathbb{R} \to \mathbb{R}$ given by
$$
f_{\mathcal{R}}(t)=
\left\{
\begin{array}{ll}
 |t|^{q-2}t, &\quad |t| \leq \mathcal{R}, \\
 \mathcal{R}^{q-\bar{q}}|t|^{\bar{q}-2}t, &\quad |t| \geq \mathcal{R}.
\end{array}
\right.
$$
Using the function $f_{\mathcal{R}}$, we
then contemplate the following auxiliary problem
\begin{equation}\label{mainequation5}
 \left\{
   \begin{array}{ll}
   \displaystyle -  \Delta_p u =\lambda |u|^{p-2}u+|u|^{p-2}u\log|u|^p+\mu f_{\mathcal{R}}(u)
~\text{in}~\R^N, \\
    \displaystyle     \int_{\R^N}|u|^pdx=a^p,
   \end{array}
 \right.
\end{equation}
 whose energy functional $J_\mathcal{R}:X\to \R$ is given by
$$
J_\mathcal{R}(u)=\frac{1}{p}\int_{\mathbb{R}^{N}}\left(|\nabla u |^{p}+|u |^{p}\right)dx
-\frac{1}{p}\int_{\mathbb{R}^{N}}|u |^{p}\log|u |^{p}dx
- \mu \int_{\mathbb{R}^{N}}F_\mathcal{R}(u) dx,
$$
where and in the sequel $F_\mathcal{R}(t)=\int_{0}^{t}f_\mathcal{R}(s)ds$.
Due to the definition of $f_{\mathcal{R}}$, one has that
\begin{equation}\label{ff}
 |f_{\mathcal{R}}(t)|\leq \mathcal{R}^{q-\bar{q}}|t|^{\bar{q}-1},~\forall t\in\R.
\end{equation}
When the constant $\mathcal{R}>0$ is fixed, it follows from \eqref{ff} that
 $f_{\mathcal{R}}$ has a $L^p$-subcritical growth because $p<\bar{q}< p+\frac{p^2}N$.
 In other words, adopting Theorem \ref{maintheorem3},
 we immediately have the results below.
\begin{corollary}\label{corollary}
Suppose that $2\leq p<N$, $\mu>0$ and $p<\bar{q}< p+\frac{p^2}N$. Then,
for every fixed $\mathcal{R}>0$,
there is an $a^*>0$ independent of $\mathcal{R}$ and
$\mu$ such that, for all fixed $a>a^*$,
Problem \eqref{mainequation5} admits a couple of weak solution
 $({u}_{\mathcal{R}}^*,{\lambda}_{\mathcal{R}}^*)\in X\times\R$ with $u_{\mathcal{R}}^*(x)>0$ for all $x\in\R^N$.
\end{corollary}

With Corollary \ref{corollary} in hands,
the reader is invited to see that if $u^*_\mathcal{R }\in X$ is a solution
of Problem \eqref{mainequation5} with $|u^*_\mathcal{R}|_\infty \leq \mathcal{R}$,
then $u^*_\mathcal{R}$ is a solution for Problem \eqref{mainequation4}. Have this in mind, our main goal
is to deduce that given an $\mathcal{R}>0$, there are $\underline{a}^*>0$ (independent of $\mathcal{R}$ and
$\mu$) and $\underline{\mu}^*=\underline{\mu}^*(\mathcal{R})>0$ such that if $a>\underline{a}^*$
 and $\mu\in (0,\underline{\mu}^*)$, then $|u^*_\mathcal{R}|_\infty \leq \mathcal{R}$.

We remark that, d'Avenia, Montefusco and Squassina \cite{d'Avenia}
handled the existence of infinitely many solutions for a class of
logarithmic Schr\"{o}dinger equations. The authors pointed out their
multiplicity results are also adapted, using \cite{ddb}, to the following logarithmic $p$-Laplacian equation
\begin{equation}\label{Aveniaequation}
-\Delta_pu=\lambda|u|^{p-2}u+|u|^{p-2}u\log |u|^p,~ u\in W^{1,p}(\mathbb{R}^n),
\end{equation}
 where $\lambda\in\R$ is a fixed constant. On the other hand here we mainly focus on existence
 of families of solutions concentrating around local minima of $V$ in the semiclassical limit $\varepsilon\to 0$.

Again the results in Theorems \ref{maintheorem2}, \ref{maintheorem3} and \ref{maintheorem4}
 are new under the $p$-Laplacian settings
with the logarithmic nonlinearity.
The striking novelty is the correct setting of functional space
in which we can treat the problems variationally.
Unfortunately, we cannot deal with the case $1<p<2$ so far ant it remains open,
see Lemma \ref{F1} below.
In addition, there are some other technical calculations due to
the $p$-Laplacian operator in the proofs of the main results.

The outline of the paper is organized as follows. In Section \ref{Preliminaryresults}, we mainly exhibit some preliminary
results. Sections \ref{semiclassical} and \ref{autonomous} are devoted to the non-autonomous
and autonomous logarithmic $p$-Laplacian equations, respectively. Finally, there are some further comments in Section \ref{comments}.
\\\\
 \textbf{Notations.} From now on in this paper, otherwise mentioned, we use the following notations:
\begin{itemize}
	\item   $C,C_1,C_2,...$ denote any positive constant, whose value is not relevant.
	 	\item      Let $(Z,\|\cdot\|_Z)$ be a Banach
space with its dual space $(Z^{*},\|\cdot\|_{Z^{*}})$.
\item  $|\cdot|_p$ denotes the usual norm of the Lebesgue measurable space in $\R^N$, for all $p \in [1,+\infty]$.
 \item $o_{n}(1)$ denotes the real sequence with $o_{n}(1)\to 0$
 as $n \to +\infty$.
\item $``\to"$ and $``\rightharpoonup"$ stand for the strong and
 weak convergence in the related function spaces, respectively.
\end{itemize}

\section{Variational setting and preliminaries}\label{Preliminaryresults}

In this section, we would like to  recommend some preliminary results.
First of all, let us introduce
some fundamental concepts and properties concerning the Orlicz
spaces. For the more details, please refer to \cite{RR} for example.

\begin{definition}\label{definition}
  An $N$-function is a continuous function $\Phi:\mathbb{R}\to[0,+\infty)$ that satisfies the following
 conditions:
\begin{itemize}
   \item[(i)] $\Phi$ is a convex and even function;
   \item[(ii)] $\Phi( t) = 0\Longleftrightarrow t= 0$;
   \item[(iii)] $\lim\limits_{t\to 0}\frac{\Phi(t)}t= 0$ and $\lim\limits_{t\to \infty}\frac{\Phi(t)}t=+\infty$.
 \end{itemize}
\end{definition}

 We say that an $N$-function $\Phi$ satisfies the $\Delta_{2}$-condition, denoted by $\Phi\in(\Delta_2)$, if
 $$
\Phi(2t)\leq k\Phi(t),~\forall t\geq t_0,
$$
for some constants $k> 0$ and $t_0\geq 0$.

The conjugate function $\tilde{\Phi}$ associated with $\Phi$ is obtained through the Legendre's transformation, defined as
$$
\tilde{\Phi}(s)=\max_{t\geq0}\{st-\Phi(t)\},~\text{for}\:s\geq0.
$$
It can be shown that that $\tilde{\Phi}$ is also an $N$-function. The functions $\Phi$ and $\tilde{\Phi}$ are mutually complementary that is,
$\tilde{\tilde{\Phi}}=\Phi$.

For an open set $\Omega\subset\mathbb{R}^N$, we define the Orlicz space associated with the $N$-function $\Phi$ as follows
$$
L^\Phi(\Omega)=\left\{u\in L_{\mathrm{loc}}^1(\Omega):\int_{\Omega}\Phi\left(\frac{|u|}{\lambda}\right)dx<+\infty,~\text{for some}\:\lambda>0\right\},
$$
which is a Banach space endowed with the Luxemburg norm given by
$$
\|u\|_{\Phi}=\inf\left\{\lambda>0:\int_{\Omega}\Phi\left(\frac{|u|}{\lambda}\right)dx\leq1\right\}.
$$
Associated with the Orlicz Spaces, there also holds the H\"{o}lder and Young type inequalities, namely
$$
st\leq\Phi(t)+\tilde{\Phi}(s),\quad\forall s,t\geq0
$$
and
$$
\left|\int_{\Omega}uvdx\right|\leq2\|u\|_{\Phi}\|v\|_{\tilde{\Phi}},\mathrm{~for~}\forall u\in L^{\Phi}(\Omega)\mathrm{~and~}\forall v\in L^{\bar{\Phi}}(\Omega).
$$
The space $L^\Phi(\Omega)$ is reflexive and separable provided that
$\Phi,\tilde{\Phi}\in(\Delta_2)$. Moreover, the $\Delta_{2}$-condition implies that
$$
L^\Phi(\Omega)=\left\{u\in L_{\mathrm{loc}}^1(\Omega):\int_{\Omega}\Phi(|u|)dx<+\infty\right\}
$$
and
$$
u_{n}\rightarrow u\mathrm{~in~}L^{\Phi}(\Omega)\Longleftrightarrow\int_{\Omega}\Phi\left(|u_{n}-u|\right)dx\rightarrow0.
$$
We then recall an significant relation involving $N$-functions that will be adopted later. Let $\Phi$ be
an $N$-function of $\mathcal{C}^1$ class and $\tilde{\Phi}$ is its conjugate function. Suppose that
\begin{equation}\label{1}
1<l\leq\frac{\Phi^{\prime}(t)t}{\Phi(t)}\leq m,\:t\neq0,
\end{equation}
then $\Phi,\tilde{\Phi}\in(\Delta_{2})$.
Finally, we consider the functions
$$
\xi_{0}(t)=\min\{t^{l},t^{m}\}\mathrm{~and~}\xi_{1}(t)=\max\{t^{l},t^{m}\},\:t\geq0,
$$
it is possible to verify that, using \eqref{1}, the function $\Phi$ satisfies the inequality below
\begin{equation}\label{2}
\xi_{0}\left(\|u\|_{\Phi}\right)\leq\int_{\mathbb{R}^{N}}\Phi(u)\leq\xi_{1}\left(\|u\|_{\Phi}\right),\:\forall u\in L^{\Phi}(\Omega).
\end{equation}

Inspired by \cite{Alves1,Alves2,Alves3}, we define the functions $F_1$ and $F_2$ as follows
$$
\left.F_1 (s)=\left
\{\begin{array}{lr}
F_1 (-s), & s\leq0,\\
-\frac{1}{p}s^p\log s^p, &0<s<(p-1)\delta,\\
 -\frac{1}{p}s^p\big[\log\big((p-1)\delta\big)^p+p+1\big]+p\delta s^{p-1}-\frac{1}{p(p-1)}\big((p-1)\delta\big)^p,&s\geq(p-1)\delta,
\end{array}\right.\right.
 $$
and
$$
\left.F_2(s)=\left\{\begin{array}{ll}0,&|s|\le(p-1)\delta,\\
\frac{1}{p}|s|^p\log\big(|s|^p/((p-1)\delta)^p\big)+p\delta|s|^{p-1}-\frac{p+1}{p}|s|^p-\frac{1}{p(p-1)}\big((p-1)\delta\big)^p,&|s|\ge(p-1)\delta,
\end{array}\right.\right.
$$
where $\delta>0$ is sufficiently small but fixed,
then we reach the decomposition \eqref{decomposition}.
Moreover, $F_1$ and $F_2$ satisfy the following properties:
\begin{itemize}
  \item[$(\textbf{P}_1)$] $F_1$ is even with $F_1^{\prime}(s)s\geq0$ and $F_1(s)\geq0$ for all $s\in\mathbb{R}$.
   Furthermore, $F_1\in \mathcal{C}^1(\mathbb{R},\mathbb{R})$ is convex if $\delta\approx0^+$;
  \item[$(\textbf{P}_2)$] $F_{2}\in \mathcal{C}^{1}( \mathbb{R},\mathbb{R})\cap \mathcal{C}^{2}((\delta,+\infty),\mathbb{R})$ and for each $\tilde{q}\in(p,p^{*})$,
  there exists a $C_{\tilde{q}}>0$ such that
$$
|F_2^{\prime}(s)|\leq C_{\tilde{q}}|s|^{\tilde{q}-1},\quad\forall s\in\mathbb{R};
$$
  \item[$(\textbf{P}_3)$] $s\mapsto\frac{F_2^{\prime}(s)}{s^{p-1}}$ is a nondecreasing function for $s>0$ and a strictly increasing function for $s>\delta$;
  \item[$(\textbf{P}_4)$] $\operatorname*{lim}\limits_{s\to\infty}\frac{F_{2}^{\prime}(s)}{s^{p-1}} =\infty$.
\end{itemize}

As a counterpart of the results explored in \cite{Alves1,Alves2,Alves3},
we conclude the following result which is nontrivial in contrast to the cited papers.

 \begin{lemma}\label{F1}
 The function $F_{1}$ is an $N$-function. Moreover, if $2\leq p<N$,
  it holds that $F_{1},\tilde{F}_{1}\in(\Delta_{2})$.
\end{lemma}

 \begin{proof} Exploiting some elementary calculations, one could easily certify that $F_1$ satisfies (I)-(III)
 of Definition \ref{definition}. To arrive at the proof, we shall verify that $F_1$ satisfies a similar relation in \eqref{1}
and so it reveals that $F_{1},\tilde{F}_{1}\in(\Delta_{2})$. Firstly, we see that
 $$
\left.F_1^\prime(s)=\left
\{\begin{array}{lr}
-(1+\log s^p)s^{p-1},&0<s<(p-1)\delta,\\
-s^{p-1}\big[\log\big((p-1)\delta\big)^p+p+1\big]+p(p-1)\delta s^{p-2},&s\geq(p-1)\delta.
\end{array}\right.\right.
 $$
 Next, we shall analyze the cases $0<s<( p- 1) \delta$ and $s\geq( p- 1) \delta$ separately.
\vskip3mm
\textbf{Case 1.}
$0< s< ( p- 1) \delta$.
\vskip3mm
In this case, it is simple to calculate that
\[
\frac{F_{1}^{\prime}(s)s}{F_{1}(s)}=p+\frac{1}{\log s},
\]
which indicates that there is an $l_1 > 1$ such that
$$
1<l_{1}\leq\frac{F_{1}^{\prime}(s)s}{F_{1}(s)}\leq m_{1}\triangleq\sup_{0<s<\delta}\left(p+\frac{1}{\log s}\right)\leq p,
$$
for some sufficiently small $\delta > 0$.
\vskip3mm
\textbf{Case 2.} $s\geq ( p- 1) \delta$.
\vskip3mm
In this case, we continue to calculate that
$$
\frac{F_1'(s)s}{F_1(s)}=\frac{-s^p\big[\log\big((p-1)\delta\big)^p+p+1\big]+p(p-1)\delta s^{p-1}}
{-\frac{1}{p}s^p\big[\log\big((p-1)\delta\big)^p+p+1\big]+p\delta s^{p-1}-\frac{1}{p(p-1)}\big((p-1)\delta\big)^p}
$$
From which, we derive that $\sup\limits_{s\geq(p-1)\delta}\frac{F_{1}^{\prime}(s)s}{F_{1}(s)}\leq p$
since for all $s\geq ( p- 1) \delta$, there holds
$$
\frac{F_{1}^{\prime}(s)s}{F_{1}(s)}
  \leq \frac{-s^p\big[\log\big((p-1)\delta\big)^p+p+1\big]+p(p-1)\delta
s^{p-1}+\big[p\delta s^{p-1}-\frac{1}{p-1}\big((p-1)\delta\big)^p\big]}{-\frac{1}{p}s^p\big[\log\big((p-1)\delta\big)^p+p+1\big]
+p\delta s^{p-1}-\frac{1}{p(p-1)}\big((p-1)\delta\big)^p}.
$$
Obviously, one can deduce that
$$\operatorname*{lim}_{s\to+\infty}\frac{F_{1}^{\prime}(s)s}{F_{1}(s)} =p\mathrm{~and~}\frac{F_{1}^{\prime}(s)s}{F_{1}(s)}>p-1,~\forall s>0,$$
and so we obtain
 $$
 p-1<\inf_{s>0}\frac{F_{1}^{\prime}(s)s}{F_{1}(s)}.$$
The last inequality together with $p>2$ guarantees the existence of an $l\in(1,2)$ such that
$$
1<l\leq\frac{F_{1}^{\prime}(s)s}{F_{1}(s)}\leq p,\:\forall s>0.
$$
Since $F_{1}$ is an even function, then the inequality holds true for any $s\neq0.$ The proof is completed.
 \end{proof}

Replacing $\Phi$ and $\Omega$ in the above discussions with $F_1$ and $\R^N$, respectively,
we conclude the Orlicz Space $L^{F_1}(\mathbb{R}^N)$
and it is standard to prove the following result.

  \begin{corollary}
The functional $\Theta:L^{F_1}(\mathbb{R}^N)\to\R$ given by $u\mapsto\int_{\mathbb{R}^N}F_1(u)dx$
is of class $\mathcal{C}^1(L^{F_1}(\mathbb{R}^N))$ with
$$
\Theta'(u)v=\int_{\mathbb{R}^N}F_1'(u)vdx,~\forall u,v\in L^{F_1}(\mathbb{R}^N),
$$
where $L^{F_1}(\mathbb{R}^N)$ denotes the Orlicz space associated with $F_1$ endowed with the Laremburg norm $\|\cdot\|_{F_1}$.
\end{corollary}

In the sequel, in order to avoid the points
$u\in W^{1,p}(\mathbb{R}^{N})$ that satisfy $F_{1}(u)\not\in L^{1}(\mathbb{R}^{N})$, we should
consider the work space $X=W^{1,p}(\mathbb{R}^{N})\cap L^{F_{1}}(\mathbb{R}^{N})$ throughout the paper
equipped with the norm
\[
\|\cdot\|\triangleq\|\cdot\|_{W^{1,p}(\mathbb{R}^N)}+\|\cdot\|_{F_1},
\]
where $\|\cdot\|_{W^{1,p}(\mathbb{R}^N)}$ denotes the usual norm in $W^{1,p}(\mathbb{R}^{N})$.
Moreover, we denote the radially symmetric subsequence of $X$ by $X_r$,
namely $X_r=\{u\in X:u(x)=u(|x|)\}$ with the norm $\|\cdot\|$.

With the space $X$ and \eqref{decomposition} in hands, we can obtain the following Br\'{e}zis-Lieb type
lemma in the logarithmic setting.

\begin{lemma}\label{BrezisLieb}
Let $\left\{u_{n}\right\}$ be a bounded sequence in $X$ such that
$u_{n}\to u$ a.e. in $\mathbb{R}^{N}$ and $\{|u_{n}|^{p}\log|u_{n}|^{p}\}$ is a bounded sequence in $L^1(\R^N)$. Then,
up to a subsequence if necessary,
$$
\begin{aligned}\lim_{n\to\infty}\int_{\mathbb{R}^N}\left(|u_n|^p\log|u_n|^p-|u_n-u|^p\log|u_n-u|^p\right)dx
=\int_{\mathbb{R}^N}|u|^p\log|u|^pdx.\end{aligned}
$$
 \end{lemma}
\begin{proof}
Recalling \eqref{decomposition}, one has that
$$
 F_2(u_n)-F_1(u_n)=\frac1p|u_n|^p\log|u_n|^p.
$$
Since $\left\{u_{n}\right\}$ is a bounded sequence in $X$, by property-$(\textbf{P}_2)$,
it follows from \cite[Lemma 1.32]{Willem} that
$$
\lim_{n\to\infty} \int_{\mathbb{R}^N}[F_2(u_n)-F_2(u_n-u)]dx=\int_{\mathbb{R}^N} F_2(u )dx.
$$
Similarly, we easily conclude that
$$
\lim_{n\to\infty} \int_{\mathbb{R}^N}[F_1(u_n)-F_1(u_n-u)]dx=\int_{\mathbb{R}^N} F_1(u )dx.
$$
So, we can finish the proof of the lemma.
\end{proof}

Next, we shall introduce the Poho\u{z}aev identity
for a class of logarithmic $p$-Laplacian equations in $\R^N$ as follows.

\begin{theorem}\label{Pohozaev}
Let $2\leq p<N$. Suppose $u\in X$ to be a nontrivial weak solution of
\begin{equation}\label{Pohozaev1}
-  \Delta_p u =\lambda |u|^{p-2}u+|u|^{p-2}u\log|u|^p+\mu |u|^{q-2}u~\text{\emph{in}}~\R^N,
\end{equation}
where $\lambda,\mu\in\R$ are constants and $p<q\leq p^*$. Then
\begin{equation}\label{Pohozaev1a}
\int_{\mathbb{R}^{N}} |\nabla u |^{p} dx=\frac{ N }{p}\int_{\mathbb{R}^{N}} |u |^{p} dx
+ \mu N\left(\frac 1p-\frac{1}{q}\right)\int_{\mathbb{R}^{N}}|u |^{q} dx.
\end{equation}
Moreover, if in addition $u(x)\geq0$ for all $x\in\R^N$, then $u(x)>0$ for all $x\in\R^N$.
\end{theorem}

\begin{proof}
The proof is divided into the following three steps.
\vskip3mm
\textbf{Step 1.} $u\in L^\infty(\mathbb{R}^N)\cap \mathcal{C}_{\mathrm{loc}}^{1,\tau}(\mathbb{R}^N)$ for some $\tau\in(0,1)$.
\vskip3mm
We start by assuming that $u\geq0$. For all $L>1$, define $u_L=\min\{u,L\}$.
Taking $\psi=u_L^{kp+1}\in X$ with $k\geq0$ as a test function in \eqref{Pohozaev1}, we obtain
\begin{equation}\label{Pohozaev2}
\int_{\R^N} |\nabla u |^{p-2}\nabla u \nabla (u_L^{kN+1}) dx
 = \int_{\R^N}[F_2'(u)-F'_1(u)+( \lambda-1)|u|^{p-2}u +\mu|u|^{q-2}u]u_L^{kp+1}dx.
\end{equation}
It is easy to observe that
 \begin{equation}\label{Pohozaev3}
 \left\{
   \begin{array}{ll}
  \displaystyle  \int_{\R^N} |\nabla u |^{p-2}\nabla u \nabla (u_L^{kp+1}) dx=\frac{kp+1}{(k+1)^p}\int_{\R^N} |\nabla (u_L)^{k+1} |^pdx,\\
 \displaystyle \int_{\R^N} |u |^{p-2}u u_L^{kp+1}dx\geq \int_{\R^N} | (u_L )^{k+1}|^pdx.
   \end{array}
 \right.
 \end{equation}
Adopting property-$(\textbf{P}_1)$ and property-$(\textbf{P}_2)$ with $\tilde{q}=q$, there holds
\begin{equation}\label{Pohozaev4}
F_2'(u)u-F'_1(u)u+ \lambda |u|^{p } +\mu|u|^{q } \leq(C_q+|\lambda|+|\mu|)|u|^{q }
\triangleq C_{q,\lambda,\mu}|u|^{q }~\text{if}~|u|\geq1.
 \end{equation}
Without loss of generality, we shall suppose that $|u|\geq1$.
Combining \eqref{Pohozaev2}, \eqref{Pohozaev3} and \eqref{Pohozaev4}, we have
\begin{align}\label{Pohozaev5}
\nonumber\bigg(\int_{\R^N}|u_L|^{(k+1)p^*}\bigg)^{\frac p{p^*}} &\leq C_{p^*}\|(u_L)^{k+1}\|^p_{W^{1,p}(\R^N)}
\leq  C_{p^*}C_{q,\lambda,\mu}(k+1)^p \int_{\R^N} |u|^{q }u ^{kp }dx \\
  & \leq C_{p^*}C_{q,\lambda,\mu}(k+1)^p\left( \int_{\R^N} |u|^{q } dx\right)^{\frac { q-p }q}
   \left( \int_{\R^N} |u| ^{(k+1)q }dx\right)^{\frac pq}
\end{align}
Letting $L\to+\infty$ in \eqref{Pohozaev5}, we arrive at
\[
\bigg(\int_{\R^N}|u_L|^{(k+1)p^*}\bigg)^{\frac p{p^*}} \leq
C_{p^*}C_{q,\lambda,\mu}(k+1)^p\mathbb{T} (u)\left( \int_{\R^N} |u| ^{(k+1)q }dx\right)^{\frac pq}
\]
which is equivalent to
\begin{equation}\label{Pohozaev6}
 \bigg(\int_{\R^N}|u|^{(k+1) p^*}\bigg)^{\frac 1{(k+1) p^*}}\leq
 C_*^{\frac1{k+1}}(k+1)^{\frac1{k+1}} \bigg(\int_{\R^N}|u|^{(k+1)q} dx\bigg)^{\frac 1{(k+1)q}},
\end{equation}
where the constant $C_*=C_{p^*}^pC_{q,\lambda,\mu}^p [\mathbb{T} (u)]^p>0$
is independent of $k$. Let $k=0$ in \eqref{Pohozaev6}, it becomes
\[
 \bigg(\int_{\R^N}|u|^{ q\varpi }\bigg)^{\frac 1{q\varpi}}\leq
 C_* \bigg(\int_{\R^N}|u|^{q} dx\bigg)^{\frac 1{q}},
\]
where $\varpi=p^*/q\geq1$.
For $k+1=\varpi^{m}$ with $m\in \mathbb{N}^+$ in \eqref{Pohozaev6}, it holds that
\[
\bigg(\int_{\R^N}|u |^{\varpi^{m+1}\sigma }dx\bigg)^{\frac{1}{\varpi^{m+1}\sigma}}
\leq C_*^{\frac1{\varpi^m}}\varpi^{\frac m{\varpi^m}}\bigg(\int_{\R^N}|u |^{\varpi^mq}dx\bigg)^{\frac{1}{\varpi^mq}}.
\]
From it, proceeding this iteration procedure $m$ times and multiplying these $m+1$ formulas,
\[
\bigg(\int_{\R^N}|u|^{\varpi^{m+1}\sigma }dx\bigg)^{\frac{1}{\varpi^{m+1}\sigma}}
\leq C_*^{\sum_{j=0}^m\frac1{\varpi^j}}\varpi^{\sum_{j=1}^m\frac j{\varpi^j}}\bigg(\int_{\R^N}|u|^{q}dx\bigg)^{\frac{1}{q}}.
\]
Since $\sum_{j=0}^\infty\frac1{\varpi^j}=\frac\varpi{\varpi-1}$ and $\sum_{j=1}^\infty\frac j{\varpi^j}=\frac{\varpi}{(\varpi-1)^2}$,
then we could take the limit as $m\to+\infty$ to conclude that $u\in L^\infty(\R^N)$.
When $u$ changes sign, then it is enough to argue as before by contemplating once the positive part $u^+ \triangleq
\max\{u, 0\}$ and once the negative part $u^- \triangleq\max\{-u, 0\}$ in place of $u$ in the definition of $u_L$. As a
result, we shall finish the verification of $u\in L^\infty(\R^N)$ for all nontrivial solution $u$.
In addition, we could follow \cite{DiBenedetto} to conclude that $u\in \mathcal{C}^{1,\tau}_{\text{loc}}(\R^N)$
for some $\tau\in(0,1)$.
\vskip3mm
\textbf{Step 2.} The nontrivial solution $u\in X$ satisfies \eqref{Pohozaev1a}.
 \vskip3mm
We recall \cite[Theorem 2]{DMS} which is presented Lemma A.1 in the Appendix and
take $\mathcal{L}(x,s,\xi)=\frac1p|\xi|^p$ which
is strictly convex in the variable $\xi\in\R^N$. Let $\varphi\in \mathcal{C}_c^1(\R^N)$
be such that $0\leq\varphi\leq1$, $\varphi(x)=1$ for all $|x|\leq1$, and $\varphi(x)=0$ for all $|x|\geq2$. Define
\[
h(x)=\varphi\bigg(\frac xk\bigg)x\in \mathcal{C}^1(\R^N,\R^N),~\text{for all}~k\in \mathbb{N}^+.
\]
Note that if $h_j(x) =\varphi\big(\frac xk\big)x_j$ for $j=1,2,\cdots$, then
\[
\left\{
  \begin{array}{ll}
\displaystyle  D_ih_j(x)=D_i\varphi\bigg(\frac xk\bigg)\frac{x_j}{k}+\varphi\bigg(\frac xk\bigg)\delta_{ij},
~\text{for all}~x\in\R^N,~j=1,2,\cdots, \\
\displaystyle  \text{div}h(x)=D \varphi\bigg(\frac xk\bigg)\frac{x }{k}+N\varphi\bigg(\frac xk\bigg),
~\text{for all}~x\in\R^N,
  \end{array}
\right.
\]
where $\delta_{ij}$ denotes the Kronecker delta symbol. One also observes that
\begin{equation}\label{Pohozaev7}
\bigg|D_i\varphi\bigg(\frac xk\bigg)\frac{x_j}{k}\bigg|\leq C,
~\text{for all}~x\in\R^N,~i,j=1,2,\cdots.
\end{equation}
Denoting $f(s)=\lambda |s|^{p-2}s+|u|^{p-2}u\log|s|^p+\mu |s|^{q-2}s$ for all $s\in\R$, by means of
\eqref{Appendix1} below, it holds that
\[
\begin{gathered}
\sum_{i,j=1}^N\int_{\R^N}D_i\varphi\bigg(\frac xk\bigg)\frac{x_j}{k} D_{\xi_i}\mathcal{L}(x,u,\nabla u)D_judx
+\int_{\R^N}\varphi\bigg(\frac xk\bigg) D_{\xi}\mathcal{L}(x,u,\nabla u)\cdot \nabla udx\hfill\\
\ \ \ \  \ \ \ \  \ \ \ \ -\int_{\R^N}\bigg[D \varphi\bigg(\frac xk\bigg)\frac{x }{k}\mathcal{L}(x,u,\nabla u)
+N\varphi\bigg(\frac xk\bigg)\mathcal{L}(x,u,\nabla u)  \bigg]dx \hfill\\
\ \ \ \  =\int_{\R^N}\bigg[\varphi\bigg(\frac xk\bigg)x\cdot \nabla u\bigg]f(u)dx.
\hfill\\
\end{gathered}
\]
Thanks to \eqref{Pohozaev7}, $\varphi\big(\frac xk\big)\to1$ and $\nabla\varphi\big(\frac xk\big)\cdot \frac xk\to0$
as $k\to+\infty$. Thus, we obtain
\[
\begin{gathered}
\sum_{i,j=1}^N\int_{\R^N}D_i\varphi\bigg(\frac xk\bigg)\frac{x_j}{k} D_{\xi_i}\mathcal{L}(x,u,\nabla u)D_judx
+\int_{\R^N}\varphi\bigg(\frac xk\bigg) D_{\xi}\mathcal{L}(x,u,\nabla u)\cdot \nabla udx\hfill\\
\ \ \ \  \ \ \ \  \ \ \ \ -\int_{\R^N}\bigg[D \varphi\bigg(\frac xk\bigg)\frac{x }{k}\mathcal{L}(x,u,\nabla u)
+N\varphi\bigg(\frac xk\bigg)\mathcal{L}(x,u,\nabla u)  \bigg]dx \hfill\\
\ \ \ \  \to\int_{\R^N} |\nabla u|^pdx-N\int_{\R^N} \frac1p|\nabla u|^pdx
=-\frac{N-p}{p}\int_{\R^N} |\nabla u|^p dx
\hfill\\
\end{gathered}
\]
as $k\to+\infty$. On the other hand, since $F(u)\in L^1(\R^N)$ for all $x\in X$ by \eqref{decomposition},
 we shall exploit an integration by parts and the Lebesgue's Dominated Convergence theorem to reach
\begin{align*}
\int_{\R^N}\bigg[\varphi\bigg(\frac xk\bigg)x\cdot \nabla u\bigg]f(u)dx&=
-N\int_{\R^N}F(u) \varphi\bigg(\frac xk\bigg)dx-\int_{\R^N}\bigg[\nabla\varphi\bigg(\frac xk\bigg)\cdot \frac xk \bigg]F(u)dx \\
  & \to -N\int_{\R^N}F(u) dx
\end{align*}
 as $k\to+\infty$. So, we can conclude the equality
\begin{equation}\label{Pohozaev8}
 \frac{N-p}{p}\int_{\R^N} |\nabla u|^p dx
 =N\int_{\R^N}\left(\frac{\lambda-1}{p} |u|^{p } +\frac1p|u|^{p} \log|u|^p+\frac{\mu }{q}|s|^{q }\right)dx.
\end{equation}
 Multiplying the nontrivial solution $u\in X$ on both sides of Eq. \eqref{Pohozaev1}, one has that
\begin{equation}\label{Pohozaev9}
 \int_{\R^N} |\nabla u|^p dx=
 \int_{\R^N}(\lambda |u|^{p }+|u|^{p } \log|u|^p+\mu |u|^{q })dx.
\end{equation}
 By multiplying $\frac Np$ in \eqref{Pohozaev9} and then minus \eqref{Pohozaev8}, we get the desired identity
 \eqref{Pohozaev1a}.

\vskip3mm
 \textbf{Step 3.} If the nontrivial solution $u(x)\geq0$ for all $x\in\R^N$, then $u(x)>0$ for all $x\in\R^N$.
 \vskip3mm

Choosing a sufficiently small $\epsilon>0$, we have
$$
\Delta_pu=-\lambda u^{p-1} -u^{p-1} \log u^p-\mu u^{q-1} \leq \xi(u)~\text{in}~
\{x\in\R^N:0<u(x)<\epsilon\},
$$
where $\xi(0)\triangleq \lim\limits_{s\to0^+}\xi(s)=0$ and $\xi:(0,+\infty)\to\R$ is defined by
$$
\xi(s)=
\left\{
  \begin{array}{ll}
   - \lambda s^{p-1} -s^{p-1} \log s^p, & \text{if}~\mu>0, \\
   -( \lambda+\mu) s^{p-1} +s^{p-1} \log s^p, & \text{if}~\mu\leq0.
  \end{array}
\right.
$$
Clearly, $\xi$ is continuous and nondecreasing when $s>0$ is small enough.
It is simple to calculate that
$\xi(\sqrt[p]{e^{-\lambda}})=0$ if $\mu>0$, and
  $\xi(\sqrt[p]{e^{-(\lambda+\mu)}})=0$ if $\mu\leq0$.
Since $u(x)\geq0$ for all $x\in\R^N$, then we apply the Step 1
and \cite[Theorem 5]{Vazquez} to finish the proof.
\end{proof}

\section{The semiclassical problem}\label{semiclassical}

In this section, we shall contemplate the existence and concentration behavior
of positive normalized solutions for a class of $p$-Laplacian
equations with logarithmic nonlinearities.
Nevertheless, first of all, let us consider
 the existence of positive solutions to the problem
\begin{equation}\label{limitproblem}
 \left.\left\{\begin{array}{l}
  \displaystyle-\Delta_pu+\mu|u|^{p-2}u=\lambda|u|^{p-2}u+|u|^{p-2}u\log|u|^p,~\text{in}~\mathbb{R}^N,\\
 \displaystyle\int_{\mathbb{R}^N}|u|^pdx=a^p,\end{array}\right.\right.
\end{equation}
 where $\Delta_pu =\text{div} (|\nabla u|^{p-2}\nabla u)$ denotes the usual $p$-Laplacian operator with $2\leq p < N$, $\mu\in [-1,+\infty)$ is a
fixed constant and $\lambda\in \R$ is known as the Lagrange multiplier.

In general, to solve Problem \eqref{limitproblem},
we look for critical points of the following variational functional
$$
I_\mu (u) =
\frac1p\int_{\mathbb{R}^N}\left[|\nabla u|^{p } +(\mu+1)|u|^{p}\right]dx
+\int_{\mathbb{R}^N}F_1 (u)vdx-\int_{\mathbb{R}^N}F_2 (u) dx
$$
restricted to the sphere $S(a)$ defined by
$$
S(a)=\left\{u\in X:\int_{\R^N}|u|^pdx=a^p\right\}.
$$
Recalling Lemma \ref{F1}, it follows that $(X,\|\cdot\|)$ is a reflexive and separable Banach space. Additionally,
note that the imbedding $X\hookrightarrow W^{1,p}(\mathbb{R}^N)$ and $X\hookrightarrow L^{F_1}(\mathbb{R}^N)$ are continuous.
As a consequence, we are derived from Section \ref{Preliminaryresults} that $I_\mu\in \mathcal{C}^1(X,\R)$ with
$$
\begin{aligned}I_\mu'(u)v=
\int_{\mathbb{R}^N}\left[|\nabla u|^{p-2}\nabla u\nabla v+(\mu+1)|u|^{p-2}v\right]dx
+\int_{\mathbb{R}^N}F_1'(u)vdx-\int_{\mathbb{R}^N}F_2'(u)vdx,\forall v\in X.\end{aligned}
$$

 Next, we will prove the following result for Problem \eqref{limitproblem}.

 \begin{theorem}\label{limittheorem}
 Let $2\leq p<N$. Then, there is a constant $\tilde{a}=\tilde{a}(\mu)>0$ such that
 Problem \eqref{limitproblem} has a couple solution $(u,\lambda)\in X\times\R$
 for all $a> \tilde{a}$, where $u(x)>0$ for all $x\in \R^N$ and $\lambda<0$.
 \end{theorem}

 The proof of the above theorem will be divided into several lemmas.

 \begin{lemma}\label{realnumber}  Let $2\leq p<N$,
the functional $I_{\mu}$ is coercive and bounded from below on $S(a)$.
\end{lemma}

\begin{proof}
In view of the property-$(\textbf{P}_{2})$ in Section \ref{Preliminaryresults},
for~every~fixed $\tilde{q}\in\left(p,\frac{p^{2}}{N}\right)$, there~exists~a~constant $C_{\tilde{q}}>0$ such~that
$$
|F_2'(s)|\leq C_{\tilde{q}}|s|^{\tilde{q}-1},~\forall s\in\mathbb{R}.
$$
Moreover, by the Gagliardo-Nirenberg inequality \eqref{GN},
$$
\begin{aligned}
I_{\mu}(u)& =\frac{1}{p}\int_{\mathbb{R}^{N}}\left[|\nabla u |^{p}+(\mu+1)|u |^{p}\right]dx
-\frac{1}{p}\int_{\mathbb{R}^{N}}|u |^{p}\log|u |^{p}dx  \\
&\geq \frac{1}{p}\int_{\mathbb{R}^{N}} |\nabla u |^{p}dx+
\int_{\mathbb{R}^N}F_1 (u) dx-C_q\mathbb{C}_{N,p,\tilde{q}}a^{\tilde{q}(1-\beta_{\tilde{q}})}\left(
\int_{\mathbb{R}^{N}} |\nabla u |^{p}dx\right)^{\frac{\tilde{q}\beta_{\tilde{q}}}{p}}.
\end{aligned}
$$
Since $\tilde{q}\in\left(p,\frac{p^{2}}{N}\right)$, then $\tilde{q}\beta_{\tilde{q}}<p$ by \eqref{GN2}.
Moreover, adopting \eqref{2}, we see that
$\int_{\mathbb{R}^N}F_1 (u) dx\to+\infty$ as $\|u\|_{F_1}\to\infty$.
These facts reveal the proof of this lemma.
\end{proof}

 As a direct consequence of Lemma \ref{realnumber}, the real number
 \[
 \mathcal{I}_{\mu,a}=\inf_{u\in S(a)}I_{\mu}(u)
 \]
 is well-defined. Then, we are going to establish some properties of
 $\mathcal{I}_{\mu,a}$ with respect to the parameter $\mu\in[-1,+\infty)$.

  \begin{lemma}\label{realnumber2}
 Let $2\leq p<N$, then there exists a constant $\tilde{a}=\tilde{a}(\mu)>0$ such that
$\mathcal{I}_{\mu,a}<0$ for all $a>\tilde{a}$ and $\mu\in[-1,+\infty)$.
 \end{lemma}

\begin{proof}
Given some fixed $\psi\in X\backslash\{0\}$ and $t>0$,
it follows from some simple calculations that
$$
I_{\mu}( t\psi)=\frac{t^p}{p}\int_{\mathbb{R}^{N}}\left[|\nabla \psi |^{p}+(\mu+1)|\psi|^{p}\right]dx
-\frac{t^p}{p}\int_{\mathbb{R}^{N}}|\psi|^{p}\log|\psi|^{p}dx -t^p\log t\int_{\R^N}|\psi|^pdx
 \to-\infty
$$
as $t\to+\infty$. Hence, there is a sufficiently large constant $\tilde{t}>0$ such that $$I_{\mu}( t\psi)\leq -1\text{ for all } t> \tilde{t}.$$
Then, we can choose $\tilde{a}=\tilde{t}|\psi|_p$ to reach the statement.
\end{proof}

 \begin{lemma}\label{monotone} Let $2\leq p<N$.
Fix $\mu\in[-1,+\infty)$ and let $0<a_{1}<a_{2}<+\infty$,
then $\frac{a_{1}^{p}}{a_{2}^{p}}\mathcal{I}_{\mu,a_{2}}<\mathcal{I}_{\mu,a_{1}}$.
 \end{lemma}

 \begin{proof}
 Since $I_{\mu}(u)=I_{\mu}(|u|)$ for each $u\in X$, without loss of generality, we suppose that $\{u_{n}\}\subset S(a_{1})$
 is a nonnegative minimizing sequence with respect to $\mathcal{I}_{\mu,a_{1}}$, that is,
$$
I_{\mu}\left(u_{n}\right)\to\mathcal{I}_{\mu,a_{1}},~\text{as}~n\to+\infty.
$$
 Choosing $v_{n}=\xi u_{n}$, then $v_{n}\in S(a_{2})$ for every $n\in\mathbb{N}$, where
 $\xi\triangleq\frac{a_{2}}{a_{1}}>1$. It~follows~from~some~simple calculations~that
 $$
 \mathcal{I}_{\mu,a_{2}}\leq I_{\mu}\left(v_{n}\right)=\xi^{p}I_{\mu}\left(u_{n}\right)-\frac{1}{p}\xi^{p}\log\xi^{p}\int_{\mathbb{R}^{N}}
 \left|u_{n}\right|^{p}dx=\xi^{p}I_{\mu}\left(u_{n}\right)-\frac{1}{p}a_{1}^{p}\xi^{p}\log\xi^{p}.
 $$
 Letting $n\to+\infty$ and using the fact that $\xi> 1$, there holds
 $$
 \mathcal{I}_{\mu,a_{2}}\leq\xi^{p}\mathcal{I}_{\mu,a_{1}}-\frac{1}{p}a_{1}^{p}\xi^{p}\log\xi^{p}<\xi^{p}\mathcal{I}_{\mu,a_{1}},
 $$
 that is,
 $$
 \frac{a_1^p}{a_2^p}\mathcal{I}_{\mu,a_2}<\mathcal{I}_{\mu,a_1},
 $$
 finishing the proof of this lemma.
 \end{proof}

Borrowing the ideas from \cite[Theorem 3.2]{Alves1},
we derive a compactness theorem on $S(a)$ which plays pivotal role
 in the proof of Theorem
\ref{limittheorem}.

\begin{theorem}\label{compactness}
(Compactness theorem on $S(a)$) Let $2\leq p<N$.
Suppose that $a>\tilde{a}$ and $\{u_n\}\subset S(a)$ is a minimizing sequence with respect to $\mathcal{I}_{\mu,a}$,
then, for some subsequence either\\
	\noindent $\emph{i)}$  $\{u_n\}$ is strongly convergent in $X$, \\
	or \\
	\noindent $\emph{ii)}$ There exists $\{y_n\}\subset \mathbb{R}^N$ such that the sequence
$v_n(x)=u_n(x+y_n)$ is strongly convergent to a function $v\in S(a)$ in $X$ with  $I_\mu(v)=\mathcal{I}_{\mu,a}$,
where $|y_n|\to+\infty$ along a subsequence.
\end{theorem}

\begin{proof}
Since $I_\mu$ is coercive on $S(a)$ by Lemma \ref{realnumber},
the sequence $\{u_n\}$ is bounded, and then, $u_n \rightharpoonup u$ in $X$ for some subsequence.
 If $u \not=0$ and $|u|_p=b \not=a$, we must have $b \in (0,a)$. By the Br\'{e}zis-Lieb Lemma (see e.g. \cite{Willem}),
	$$
	|u_n|_p^{p}=|u_n-u|_p^{p}+|u|_p^{p}+o_n(1).
	$$
	Setting $v_n=u_n-u$, $d_n=|v_n|_p$ and supposing that $|v_n|_p \to d$, we get $a^p=b^p+d^p$. From $d_n \in (0,a)$ for $n$ large enough,
we apply Lemma \ref{BrezisLieb} to see that
	$$
	\mathcal{I}_{\mu,a}+o_n(1)=I_{\mu}(u_n)=I_{\mu}(v_n)+I_{\mu}(u)+o_n(1)\geq \mathcal{I}_{\mu,d_n}+\mathcal{I}_{\mu,b}+o_n(1).
	$$
Thereby, by Lemma \ref{monotone},
	$$
	\mathcal{I}_{\mu,a}+o_n(1) \geq \frac{d_n^p}{a^p}\mathcal{I}_{\mu,a}+\mathcal{I}_{\mu,b}+o_n(1).
	$$
	Letting $n \to +\infty$, we find
	\begin{equation}\label{newine}
		\mathcal{I}_{\mu,a} \geq \frac{d^p}{a^p}\mathcal{I}_{\mu,a}+\mathcal{I}_{\mu,b}.
	\end{equation}
	Since $b \in (0,a)$, using again Lemma \ref{monotone} in \eqref{newine}, we get the following inequality
	$$
	\mathcal{I}_{\mu,a} > \frac{d^p}{a^p}\mathcal{I}_{\mu,a}+\frac{b^p}{a^p}\mathcal{I}_{\mu,b}
=\left(\frac{d^p}{a^p}+\frac{b^p}{a^p}\right)\mathcal{I}_{\mu,a}=\mathcal{I}_{\mu,a},
	$$
	which is absurd. This asserts that $|u|_p=a$, or equivalently, $u \in S(a)$.
As $|u_n|_p=|u|_p=a$, $u_n \rightharpoonup u$ in $L^{p}(\mathbb{R}^N)$ and $L^{p}(\mathbb{R}^N)$ is reflexive, it is well-known that
	\begin{equation} \label{(a)}
		u_n \to u~ \mbox{in}~ L^{p}(\mathbb{R}^N).
	\end{equation}
	This combined with interpolation theorem in the Lebesgue space and property-$(\textbf{P}_2)$ gives
	\begin{equation} \label{F}
		\int_{\mathbb{R}^N}F_2(u_n) dx\to \int_{\mathbb{R}^N}F_2(u) dx.
	\end{equation}
	These limits together with $\mathcal{I}_{\mu,a}=\displaystyle \lim_{n \to +\infty}I_\mu(u_n)$ and $F_1\geq0$
in property-$(\textbf{P}_1)$
indicate that
	$$
	\mathcal{I}_{\mu,a}\geq I_\mu(u).
	$$
	As $u \in S(a)$, therefore $I_\mu(u)=\mathcal{I}_{\mu,a}$, then
	$$
	\lim_{n \to +\infty}I_\mu(u_n)= I_\mu(u),
	$$
	that combines with (\ref{(a)}) and (\ref{F}) to give
$$
\lim_{n\to\infty}\int_{\mathbb{R}^N}|\nabla u_n|^p dx=\int_{\mathbb{R}^N}|\nabla u |^p  dx
$$
and
	$$
	\lim_{n\to\infty}\int_{\mathbb{R}^N}F_1(u_n) dx=\int_{\mathbb{R}^N}F_1(u) dx.
	$$
Recalling $F_1\in(\Delta_2)$ by Lemma \ref{F1},
jointly with the above two limits as well as \eqref{(a)},
it enables to see that $u_n \to u$ in $X$.
	
	Now, assume that $u=0$, that is, $u_n \rightharpoonup 0$ in $X$. We
claim that there exists $C>0$ such that
	\begin{equation} \label{EQ2}
		\int_{\mathbb{R}^N}F_2(u_n) dx\geq C,~\text{for}~ n \in \mathbb{N}~\text{large enough}.
	\end{equation}
	Otherwise, there is a subsequence of $\{u_n\}$, still denoted by itself, such that
\[
\int_{\mathbb{R}^N}F_2(u_n) dx\to0
\]
as $n\to\infty$. It follows from Lemma \ref{realnumber2} and \eqref{EQ2} that
$$
0>\mathcal{I}_{\mu,a}=\lim_{n\to\infty}I_\mu(u_n)=\lim_{n\to\infty}\left(
\frac{1}{p}\int_{\mathbb{R}^{N}} [|\nabla u_n |^{p}+(\mu+1)|u_n|^p]dx+
\int_{\mathbb{R}^N}F_1 (u_n) dx\right)\geq0
$$
which is impossible.

So, there are $R,C>0$ and $\{y_n\}\subset \mathbb{R}^N$ such that
\begin{equation} \label{EQ3}
\int_{B_R(y_n)}|u_n|^p dx\geq C,~\text{for all}~ n \in \mathbb{N}.
	\end{equation}
If it
 it is not the case, then we derive $u_n\to0$ for all $p<s<p^*$ by the Vanishing lemma
which yields that $F_2(u_n)\to0$ in $L^1(\R^N)$ by property-$(\textbf{P}_2)$, a contradiction to \eqref{EQ2}.
Recalling $u=0$, we further have that $\{y_n\}$ is unbounded in $\R^N$.
Define $v_n(x)=u_n(x+y_n)$, then $\{v_n\}\subset S(a)$ and it is also a minimizing sequence with respect to $\mathcal{I}_{\mu,a}$.
Moreover, owing to \eqref{EQ3}, passing to a subsequence if necessary, there is a $v\in X\backslash\{0\}$ such that
$$
v_n\rightharpoonup v~\text{in}~X~\text{and}~v_n\rightharpoonup v~\text{a.e. in}~\R^N.
$$
Repeating the calculations in the first part of the proof,
it must conclude that $v_n\to v$ in $X$ along a subsequence.
The proof is completed.
\end{proof}

\begin{proof}[\textbf{\emph{Proof of Theorem \ref{limittheorem}}}]
By Lemma \ref{realnumber}, there exists a bounded minimizing sequence $\{u_n\}\subset S(a)$ with respect to $\mathcal{I}_{\mu,a}$,
that is, $I_\mu(u_n) \to \mathcal{I}_{\mu,a}$.
Thanks to Theorem \ref{compactness}, there exists a $u_a\in S(a)$ with  $I_\mu(u_a)=\mathcal{I}_{\mu,a}$. Therefore, by the Lagrange multiplier
  theorem, there exists $\lambda_a\in\R$ such that
\begin{equation} \label{EQ4}
	I_\mu'(u_a)=\lambda_a \Psi'(u_a) ~ \mbox{in} ~ X^*,
\end{equation}
where $\Psi:X \to \mathbb{R}$ is given by
$$
\Psi(u)=\frac1p\int_{\mathbb{B}^N}|u|^pdx,~u \in X.
$$
Thereby, according to (\ref{EQ4}), the couple $(u_a,\lambda_a)\in S(a)\times \R$ satisfies the following equation
$$
-\Delta_pu +\mu|u |^{p-2}u =\lambda|u |^{p-2}u +|u |^{p-2}u \log|u |^p,~\text{in}~\mathbb{R}^N.
$$
Let $u_a\in S(a)$ be a test function on the both sides for the above equation, it holds that
$$
\int_{\mathbb{R}^N}\left(|\nabla u_a|^{p } +\mu|u_a|^{p}\right)dx=\lambda_aa^p
+\int_{\mathbb{R}^N}|u_a |^{p} \log|u_a |^pdx
$$
which indicates that
$$
\mathcal{I}_{\mu,a}=I_\mu(u_a)=\frac{1}{p}\int_{\mathbb{R}^{N}} |u_{a}|^{p} dx
+\frac{\lambda_a}{p}a^p\geq \frac{\lambda_a}{p}a^p.
$$
Due to Lemma \ref{realnumber2}, one sees that $\lambda_a<0$. Since $u\in S(a)$ implies that $|u|\in S(a)$ and $I_\mu(u)= I_\mu(|u|)$
 for all $u\in X$
which give that
$$|u_a|\in S(a)
~\text{and}~
\mathcal{I}_{\mu,a}=I_\mu(u_a)\geq I_\mu(|u_a|)\geq \mathcal{I}_{\mu,a}.
$$
So, we can replace $u_a$ with $|u_a|$ and then, without loss of generality, we shall
suppose that $u_a\geq0$. A very similar arguments in Step 3 in the proof of Theorem \ref{Pohozaev} shows that
 $u_a$ is positive. The proof is completed.
\end{proof}

Thanks to Theorem \ref{limittheorem}, we immediately have the following result whose detailed proof is omitted.

\begin{corollary}\label{limitcorollary}
Let $2\leq p<N$. If $a>\tilde{a}$ and $-1\leq \mu_1<\mu_2<+\infty$ are fixed, then
$\mathcal{I}_{\mu_1,a}<\mathcal{I}_{\mu_2,a}<0$.
\end{corollary}

From now on, we begin investigating the existence and concentration behavior
of positive solutions for \eqref{mainequation1}-\eqref{mainequation1a}.
To the aims, we consider the variational functional $I_\varepsilon:X\to\R$
given by
\begin{equation}\label{Ifunctional}
 I_\varepsilon (u) =
\frac1p\int_{\mathbb{R}^N}\left[|\nabla u|^{p } +(V(\varepsilon x)+1)|u|^{p}\right]dx
+\int_{\mathbb{R}^N}F_1 (u)dx-\int_{\mathbb{R}^N}F_2 (u) dx
\end{equation}
restricted to the sphere $S(a)$ and the minimization problem
$$
\mathcal{I}_{\varepsilon,a}=\inf_{u\in S(a)}I_{\varepsilon}(u).
$$

According to Theorem \ref{compactness}, it is significant
to derive a similar compactness theorem for $I_\varepsilon$ on $S(a)$.
So, we shall focus on verifying it. Let us introduce the two functionals
$I_0,I_\infty:S(a)\to\R$ defined by
$$
\left\{
  \begin{array}{ll}
\displaystyle I_0 (u) =
\frac1p\int_{\mathbb{R}^N}\left[|\nabla u|^{p } +(V_0+1)|u|^{p}\right]dx
+\int_{\mathbb{R}^N}F_1 (u)dx-\int_{\mathbb{R}^N}F_2 (u) dx,
\\
\displaystyle I_\infty (u) =
\frac1p\int_{\mathbb{R}^N}\left[|\nabla u|^{p } +(V_\infty+1)|u|^{p}\right]dx
+\int_{\mathbb{R}^N}F_1 (u)dx-\int_{\mathbb{R}^N}F_2 (u) dx.
  \end{array}
\right.
$$
The corresponding minimization problems are given by
$$
\mathcal{I}_{0,a}=\inf_{u\in S(a)}I_{0}(u)~\text{and}~
\mathcal{I}_{\infty,a}=\inf_{u\in S(a)}I_{\infty}(u).
$$
Since $-1\leq V_0<V_\infty<+\infty$ by $(\hat{V}_1)$, we are derived from Corollary \ref{limitcorollary}
that
\begin{equation}\label{5}
\mathcal{I}_{0,a}<\mathcal{I}_{\infty,a}<0~\text{for all}~a>\tilde{a}>0.
\end{equation}

\begin{lemma}\label{limitrealnumber}
Let $2\leq p<N$ and $a>\tilde{a}>0$, then $\limsup\limits_{\varepsilon\to0^+}\mathcal{I}_{\varepsilon,a}\leq \mathcal{I}_{0,a}$.
In particular, there exists a sufficiently small $\varepsilon^*>0$ such that
$\mathcal{I}_{\varepsilon,a}< \mathcal{I}_{\infty,a}$ for all $\varepsilon\in(0,\varepsilon^*)$.
\end{lemma}

\begin{proof}
Adopting Theorem \ref{limittheorem}, for all $a>\tilde{a}>0$, there is a
$u_0\in S(a)$ such that $I_{0}(u_0)=\mathcal{I}_{0,a}$.
So,
$$
\mathcal{I}_{\varepsilon,a}\leq I_{\varepsilon}(u_0)
=\frac1p\int_{\mathbb{R}^N}\left[|\nabla u_0|^{p } +(V(\varepsilon x)+1)|u_0|^{p}\right]dx
+\int_{\mathbb{R}^N}F_1 (u_0)dx-\int_{\mathbb{R}^N}F_2 (u_0) dx.
$$
Adopting the Lebesgue's theorem and taking the limit as $\varepsilon\to0^+$, there holds
$$
\limsup\limits_{\varepsilon\to0^+}\mathcal{I}_{\varepsilon,a}\leq\limsup\limits_{\varepsilon\to0^+}
I_{\varepsilon}(u_0)=I_{0}(u_0)=\mathcal{I}_{0,a}
$$
finishing the first part of the lemma. Due to \eqref{5}, one could find such a
constant $\varepsilon^*>0$ such that $\mathcal{I}_{\varepsilon,a}< \mathcal{I}_{\infty,a}$ for all $\varepsilon\in(0,\varepsilon^*)$.
The proof is completed.
\end{proof}

\begin{lemma}\label{limitrealnumber}
Let $2\leq p<N$ and $a>\tilde{a}>0$. If $\varepsilon\in(0,\varepsilon^*)$ is fixed and suppose that
$\{u_n\}\subset S(a)$ such that $I_\varepsilon(u_n)\to \hat{d}<\frac{\mathcal{I}_{0,a}+\mathcal{I}_{\infty,a}}{2}$,
then there is a $u\not=0$ such that $u_n\rightharpoonup u$ in $X$ along a subsequence.
\end{lemma}

\begin{proof}
Arguing as Lemma \ref{realnumber}, one deduces that $\{u_n\}$ is bounded in $X$.
Passing to a subsequence if necessary,
there is a $u\in X$ such that $u_n\rightharpoonup u$ in $X$
and $u_n\to u$ a.e. in $\R^N$.
To prove that $u\not=0$, let us suppose it by a contradiction
and assume that $u\equiv0$. Hence,
$$
\hat{d}+o_n(1)=I_\varepsilon(u_n)=I_\infty(u_n)
+\frac1p\int_{\R^N}[V(\varepsilon x)-V_\infty]|u_n|^pdx.
$$
Due to $(V_2)$, given an arbitrary $\epsilon>0$, there is an $R>0$ such that
$$
V(x)\geq V_\infty-\epsilon,\text{ for all }|x|\geq R
$$
which indicates that
$$
\hat{d}+o_n(1)\geq I_\infty(u_n)+\frac1p\int_{B_{R/\varepsilon}(0)}[V(\varepsilon x)-V_\infty]|u_n|^pdx
-\frac\epsilon p\int_{B^c_{R/\varepsilon}(0)}|u_n|^pdx.
$$
Since $\{u_n\}$ is bounded in $X$ and $u_n\to 0$ in $L^p(B_{R/\varepsilon}(0))$, one has that
$$
\hat{d}+o_n(1)\geq I_\infty(u_n)-C\epsilon\geq \mathcal{I}_{\infty,a}-C\epsilon
$$
for some $C>0$ independent of $\epsilon$. Let us tend $\epsilon\to0^+$,
then $\hat{d}\geq \mathcal{I}_{\infty,a}$ which is impossible due to the facts that
$\hat{d}<\frac{\mathcal{I}_{0,a}+\mathcal{I}_{\infty,a}}{2}$ and \eqref{5}. The proof is completed.
\end{proof}

\begin{lemma}\label{PSnontrivial}
Let $2\leq p<N$ and $a>\tilde{a}>0$. Suppose $\varepsilon\in(0,\varepsilon^*)$ to be fixed
and let $\{u_n\}\subset X$ be a $(PS)_{\hat{d}}$ sequence for $I_\varepsilon$ constrained to $S(a)$
with $\hat{d}<\frac{\mathcal{I}_{0,a}+\mathcal{I}_{\infty,a}}{2}$, namely
$$
I_\varepsilon(u_n)\to\hat{d}~\text{and}~\|I'_\varepsilon|_{S(a)}\|_{X^*}\to0~\text{as}~n\to\infty,
$$
then, up to a subsequence if necessary, there is a $u\in X$ such that $u_n\rightharpoonup u$ in $X$.
Moreover if $u_n\not\to u$ in $X$, then there exists a $\hat{\delta}>0$ independent of $\varepsilon\in(0,\varepsilon^*)$ such that,
by decreasing $\varepsilon^*$ if necessary, there holds
$$
\liminf_{n\to\infty}\int_{\R^N}|v_n|^pdx\geq\hat{\delta}.
$$
\end{lemma}

\begin{proof}
We define the functional $\Psi:X\to\R$ by
$$
\Psi(u)=\frac1p\int_{\mathbb{B}^N}|u|^pdx,~u \in X,
$$
it follows that $S(a)=\Psi^{-1}(\{a^p/p\})$. Hence, adopting \cite[Proposition 5.12]{Willem},
there exists $\{\lambda_n\}\subset\R$ such that
$$
\|I'_\varepsilon(u_n)-\lambda_n\Psi'(u_n)\|_{X^*}\to0~\text{as}~n\to\infty.
$$
Since $\{u_n\}$ is bounded in $X$, we easily get that $\{\lambda_n\}$ is bounded in $\R$.
Passing to a subsequence if necessary, there is a $\lambda\in\R$ that may depend on $\varepsilon$
such that $\lambda_n\to \lambda$ and so
$$
\|I'_\varepsilon(u_n)-\lambda \Psi'(u_n)\|_{X^*}\to0~\text{as}~n\to\infty
$$
which immediately shows us that
$$
I'_\varepsilon(u )-\lambda \Psi'(u )=0~\text{in}~X^*.
$$
Combining Lemma \ref{BrezisLieb} and the Br\'{e}zis-Lieb lemma, one has that
\begin{align*}
I'_\varepsilon(v_n)v_n-\lambda \Psi'(v_n)v_n &=I'_\varepsilon(u_n)u_n-I'_\varepsilon(u)u -\lambda \Psi'( u_n) u_n
+\lambda \Psi'( u ) u +o_n(1) \\
    &  =I'_\varepsilon(u_n)u_n-\lambda_n \Psi'( u_n) u_n-I'_\varepsilon(u)u
+\lambda \Psi'( u ) u +o_n(1)\\
&=o_n(1)
\end{align*}
jointly with $F_1'(s)s\geq0$ for all $s\in\R$ in property-$(\textbf{P}_1)$ and property-$(\textbf{P}_2)$ implies that
$$
\int_{\R^N}[|\nabla v_n|^p+(V(\varepsilon x)+1-\lambda)|v_n|^p]dx\leq
 C_{\tilde{q}}\int_{\R^N}|v_n|^{\tilde{q}}dx+o_n(1)
$$
for some $\tilde{q}\in(p,p^*)$. We claim that there is a $\lambda^*<0$ independent of $\varepsilon\in(0,\varepsilon^*)$ such that
$$
\lambda\leq \lambda^*,~\forall \varepsilon\in(0,\varepsilon^*).
$$
Indeed, due to $\{u_n\}\subset S(a)$, we find that
$$
\hat{d}=\lim _{n\to\infty}I_\varepsilon(u_n)=\lim _{n\to\infty}
\left[I_\varepsilon(u_n)-\frac1p\bigg(I'_\varepsilon(u_n)u_n-\lambda\Psi'(u_n)u_n\bigg)\right]
\geq \frac\lambda pa^p
$$
showing the claim. As a consequence, owing to $V_0\geq-1$ by $(\hat{V}_1)$, there holds
$$
\int_{\R^N}(|\nabla v_n|^p -\lambda^* |v_n|^p)dx\leq
 C_{\tilde{q}}\int_{\R^N}|v_n|^{\tilde{q}}dx+o_n(1),~\forall \varepsilon\in(0,\varepsilon^*).
$$
It follows from the continuous imbedding $W^{1,p}(\R^N)\hookrightarrow L^q(\R^N)$ with $q\in(p,p^*)$ that
$$
\|v_n\|_{W^{1,p}(\R^N)}^p\leq C_1|v_n|^{\tilde{q}}_{\tilde{q}}+o_n(1)\leq C_2\|v_n\|_{W^{1,p}(\R^N)}^{\tilde{q}},~\forall \varepsilon\in(0,\varepsilon^*),
$$
where $C_1,C_2>0$ are independent of $\varepsilon\in(0,\varepsilon^*)$.
Because $v_n\not\to0$ in $X$ and $\tilde{q}>p$, without loss of generality, we can suppose that
\begin{equation}\label{7}
 \liminf_{n\to\infty}\|v_n\|_{W^{1,p}(\R^N)}\geq C_3,~\forall \varepsilon\in(0,\varepsilon^*),
\end{equation}
where $C_3>0$ is independent of $\varepsilon\in(0,\varepsilon^*)$.
Otherwise, if $v_n\to0$ in $W^{1,p}(\R^N)$, then $F_2'(v_n)v_n\to0$
 in $L^1(\R^N)$ by property-$(\textbf{P}_2)$ which together with $I'_\varepsilon(v_n)v_n-\lambda \Psi'(v_n)v_n=o_n(1)$
 and $F_2'(s)s\geq0$ for all $s\in\R$ by property-$(\textbf{P}_1)$ yields that $F_1(v_n)\to0$
 in $L^1(\R^N)$. Recalling Lemma \ref{F1}, we can conclude that $v_n\to0$ in $L^{F_1}(\R^N)$
 and so $v_n\to0$ in $X$, a contradiction and \eqref{7} follows. Hence,
 $$
  \liminf_{n\to\infty} |v_n |_{\tilde{q}}^{\tilde{q}}\geq C_2^{-1}C_3^p,~\forall \varepsilon\in(0,\varepsilon^*).
 $$
 Adopting \eqref{GN} and $|\nabla v_n|_p^p$ is bounded, we can derive the proof of this lemma.
 \end{proof}

\begin{theorem}\label{PScondition}
Let $2\leq p<N$ and $a>\tilde{a}>0$. If $\varepsilon\in(0,\varepsilon^*)$ is fixed,
then the functional $I_\varepsilon|_{S(a)}$ satisfies the $(PS)_{\hat{d}}$ condition
with $\hat{d}<\mathcal{I}_{0,a}+\Upsilon$, where $0<\Upsilon\leq\min\{\frac12,\frac{\hat{\delta}}{a^p}\}(\mathcal{I}_{\infty,a}-\mathcal{I}_{0,a})$
and $\hat{\delta}>0$ is determined by Lemma \ref{PSnontrivial}.
\end{theorem}

\begin{proof}
Let $\{u_n\}\subset X$ be a $(PS)_d$ sequence for $I_\varepsilon|_{S(a)}$
and define the functional $\Psi:X\to\R$ by
$$
\Psi(u)=\frac1p\int_{\mathbb{B}^N}|u|^pdx,~u \in X,
$$
which reveals that $S(a)=\Psi^{-1}(\{a^p/p\})$. Thus, adopting \cite[Proposition 5.12]{Willem},
there exists $\{\lambda_n\}\subset\R$ such that
$$
\|I'_\varepsilon(u_n)-\lambda_n\Psi'(u_n)\|_{X^*}\to0~\text{as}~n\to\infty.
$$
According to Lemma \ref{realnumber2}, $\{u_n\}$ is bounded in $X$
and then there is a $u\in X$ such that $u_n\rightharpoonup u$ in $X$
and $u_n\to u$ a.e. in $\R^N$ along a subsequence.
Denoting $v_n\triangleq u_n-u$, if $v_n\not\to0$ in $X$, then
 Lemma \ref{PSnontrivial} ensures that
\begin{equation}\label{6}
\liminf_{n\to\infty}\int_{\R^N}|v_n|^pdx\geq\hat{\delta}.
\end{equation}
Let $d_n=|v_n|_p$ and $|u|_p= b$, so we can assume that $|v_n|_p\to d> 0$ and $b>0$
by \eqref{6} and Lemma \ref{limitrealnumber}, respectively.
Moreover, it holds that $a^p=b^p+d^p$ by the Br\'{e}zis-Lieb lemma.
Via exploiting a similar argument explored in Lemma \ref{limitrealnumber},
by means of $v_n\rightharpoonup 0$ in $X$, one can show that
$I_\varepsilon(v_n)\geq \mathcal{I}_{\infty,d_n} +o_n(1)$ which together
with Lemma \ref{BrezisLieb} and $V(x)\geq V_0$ for all $x\in\R^N$  by $(\hat{V}_1)$ gives that
$$
\hat{d}+o_n(1) =I_\varepsilon(u_n)=I_\varepsilon(v_n)+I_\varepsilon(u )+o_n(1)\geq \mathcal{I}_{\infty,d_n}
  +\mathcal{I}_{0,b}+o_n(1).
$$
Since $d_n\in (0,a)$ for $n\in \mathbb{N}$ large enough and $b\in(0,a)$,
we argue as Lemma \ref{monotone} to derive
$$
\hat{d}+o_n(1)\geq \frac{d_n^p}{a^p}\mathcal{I}_{\infty,a}+\frac{b^p}{a^p}\mathcal{I}_{0,a}
$$
Letting $n\to\infty$ and recalling $a^p=b^p+d^p$, one has that
$$
\Upsilon>\frac{\hat{\delta}}{a^p}(\mathcal{I}_{\infty,a}-\mathcal{I}_{0,a})
$$
violating to the definition of $\Upsilon$. So, we must conclude that $v_n\to 0$ in $X$,
that is $u_n\to u$ in $X$. One further obtains that $u\in S(a)$ and
$$
-  \Delta_p u+V(\varepsilon x)|v|^{p-2}u=\lambda |u|^{p-2}v+|u|^{p-2}v\log|v|^p~\text{in}~\R^N,
$$
where $\lambda\leq\lambda^*<0$ comes from Lemma \ref{PSnontrivial}.
The proof is completed.
\end{proof}

In what follows, according to $(V_2)$,
 we fix some sufficiently small $\rho_0,r_0>0$ to satisfy
 \begin{itemize}
   \item $\overline{B_{\rho_{0}}(x^{i})}\cap\overline{B_{\rho_{0}}(x^{j})}=\emptyset$ for $i\neq j$ and $i,j\in\{1,\ldots,\ell\}$;
   \item $\bigcup\limits_{i=1}B_{\rho_0}(x^i)\subset B_{r_0}(0)$;
   \item $K_{\frac{\rho_0}2}=\bigcup\limits_{i=1}^{\ell}\overline{B_{\frac{\rho_0}2}(x^i)}$.
 \end{itemize}
Define the function $Q_\varepsilon:X\backslash\{0\}\to\R$ by
$$
Q_\varepsilon(u)=\frac{\displaystyle\int_{\mathbb{R}^N}\chi(\varepsilon x)|u|^pdx}{\displaystyle\int_{\mathbb{R}^N}|u|^pdx},
$$
where $\chi:\mathbb{R}^N\to\mathbb{R}^N$ is given by
\[
\chi(x)=\begin{cases}x,&\text{if}~|x|\le r_0,\\\frac{r_0x}{|x|},&\text{if}~|x|>r_0.\end{cases}
\]

With Theorem \ref{PScondition} in hands, we now focus on establishing the
existence of $(PS)$ sequences for the variational functional $I_\varepsilon$ constrained on $S(a)$.

\begin{lemma}\label{PS1}
Let $2\leq p<N$ and $a>\tilde{a}>0$. Suppose $\varepsilon\in(0,\varepsilon^*)$ to be fixed, decreasing $\varepsilon^*>0$
if necessary,
there is $\delta^*>0$ such that if $u\in S(a)$ and $I_{\varepsilon}(u)\leq \mathcal{I}_{0,a}+\delta^*$,
then
$$
Q_\varepsilon(u)\in K_{\frac{\rho_0}2},~\forall \varepsilon\in(0,\varepsilon^*).
$$
\end{lemma}

\begin{proof}
Suppose, by contradiction, that for all $n\in \mathbb{N}$, there exist $\varepsilon_n\to0$ and $\{u_n\}\subset S(a)$ such that
$$\mathcal{I}_{0,a}\leq I_{\varepsilon_n}(u_n)\leq \mathcal{I}_{0,a}+\frac1n,$$
and
$$
Q_{\varepsilon_n}(u_n)\not\in K_{\frac{\rho_0}2}.
$$
Obviously, one has that
$$\mathcal{I}_{0,a}\leq I_{0}(u_n)
\leq I_{\varepsilon_n}(u_n)\leq \mathcal{I}_{0,a}+\frac1n$$
showing that $\{u_n\}\subset S(a)$  is a minimizing sequence with respect to $\mathcal{I}_{0,a}$.
Thanks to Theorem \ref{compactness}, passing to a subsequence if necessary,
one of the following alternatives holds true

i) There is a function $u\in S(a)$ such that $u_n\to u$ in $X$ as $n\to\infty$;

ii) There exists a sequence $\{y_n\}\subset\R^N$ with $|y_n|\to+\infty$ such that $v_n=u_n(\cdot+y_n)\to v$ in $X$
for some $v\in S(a)$.

\underline{\textbf{We claim that i) cannot occur.}} Otherwise, adopting the definition of $\chi$, one has
$$
\lim_{n\to\infty}\int_{\mathbb{R}^N}\chi(\varepsilon_n x)|u_n|^pdx
=\lim_{n\to\infty}\int_{\mathbb{R}^N} \chi(0) |u|^pdx =0.
$$
From which, we conclude that $Q_{\varepsilon_n}(u)\in K_{\frac{\rho_0}2}$ for some sufficiently large $n\in \mathbb{N}$.
It is impossible and so the claim follows.
When ii) occurs, passing to a subsequence if necessary, we shall contemplate the following
two cases:

\underline{\textbf{ii)-(1). $|\varepsilon_ny_n|\to+\infty$ as $n\to\infty$.}}\\
In this case, as a consequence of $v_n\to v$ in $X$, there holds
 \begin{align*}
   I_{\varepsilon_n}(u_n) & =
\frac1p\int_{\mathbb{R}^N}\left[|\nabla u_n|^{p } +(V(\varepsilon_n x)+1)|u_n|^{p}\right]dx
+\int_{\mathbb{R}^N}F_1 (u_n)dx-\int_{\mathbb{R}^N}F_2 (u_n) dx \\
    &=
\frac1p\int_{\mathbb{R}^N}\left[|\nabla v _n|^{p } +(V(\varepsilon_n x+\varepsilon_n y_n)+1)|v_n|^{p}\right]dx
+\int_{\mathbb{R}^N}F_1 (v_n)dx-\int_{\mathbb{R}^N}F_2 (v_n) dx\\
&\to I_{\infty }(v).
 \end{align*}
 Since $I_{\varepsilon_n}(u_n)\leq \mathcal{I}_{0,a}+\frac1n$, we arrive at the inequality below
 $$
 \mathcal{I}_{0,a}\geq I_{\infty }(v)\geq \mathcal{I}_{\infty,a}
 $$
 which contradicts with \eqref{5}.

\underline{\textbf{ii)-(2). $\varepsilon_ny_n\to y$ for some $y\in\R^N$ as $n\to\infty$.}}\\
In this case, a similar argument using the above calculations indicates that
$$
\mathcal{I}_{V(y),a}\leq \mathcal{I}_{ 0,a}.
$$
If $V(y)>V_0$, we follow a very similar approach explored in the proof of Corollary \ref{limitcorollary}
to deduce that $\mathcal{I}_{V(y),a}> \mathcal{I}_{ 0,a}$ which is absurd. Thereby, $V(y)=V_0$
and $y=x^i$ for some $i\in\{0,1,\cdots,l\}$.
Then one derives that
$$
\lim_{n\to\infty}\int_{\mathbb{R}^N}\chi(\varepsilon_n x)|u_n|^pdx
=\lim_{n\to\infty}\int_{\mathbb{R}^N} \chi(\varepsilon_n x+\varepsilon_n y_n) |v_n|^pdx =x^i\int_{\mathbb{R}^N} |v|^pdx
$$
which reveals that $\lim\limits_{n\to\infty}Q_{\varepsilon_n}(u_n)=x^i\in K_{\frac{\rho_0}{2}}$. From which,
we obtain that $Q_{\varepsilon_n}(u_n)\in K_{\frac{\rho_0}{2}}$ for some sufficiently large $n\in N$, a contradiction.
The proof is completed.
\end{proof}

In the sequel, for  $j\in\{1,\cdots,l\}$, we need the following notations
$$
\begin{aligned}
&\bullet\theta_{\varepsilon}^{j}=\{u\in S(a):|Q_{\varepsilon}(u)-x^{j}|<\rho_{0}\}, \\
&\bullet\partial\theta_{\varepsilon}^{j}=\{u\in S(a):|Q_{\varepsilon}(u)-x^{j}|=\rho_{0}\}, \\
&\bullet\beta_\varepsilon^j=\operatorname*{inf}_{u\in\theta_\varepsilon^j}I_\varepsilon(u) ~
  \mathrm{and} ~ {\tilde{\beta}_\varepsilon^j}=\inf_{u\in\partial\theta_\varepsilon^j}I_\varepsilon(u).
\end{aligned}
$$

\begin{lemma}\label{PS2}
Let $2\leq p<N$ and $a>\tilde{a}>0$. Suppose $\varepsilon\in(0,\varepsilon^*)$ to be fixed, decreasing $\varepsilon^*>0$
if necessary, for the constant $\Upsilon>0$ in Theorem \ref{PScondition},
there holds
$$
\beta_\varepsilon^j<\mathcal{I}_{0,a}+\Upsilon
~
  \mathrm{and} ~
\beta_\varepsilon^j<\tilde{\beta}_\varepsilon^j,~\forall \varepsilon\in(0,\varepsilon^*).
$$
\end{lemma}

\begin{proof}
According to Theorem
\ref{limittheorem}, there is a function $u\in S(a)$ such that
$$I_0(u)=\mathcal{I}_{0,a}
~\text{and}~I'_0(u)=0~\text{in}~X^*.
$$
For $j\in\{0,1,\cdots,l\}$, we define the function $\hat{u}_\varepsilon^j:\R^N\to\R$ by
$$
\hat{u}_\varepsilon^j=u(x-x^j/\varepsilon).
$$
By a simple change of variable, one has that
$$
I_{\varepsilon}(\hat{u}_\varepsilon^j)  =
\frac1p\int_{\mathbb{R}^N}\left[|\nabla u |^{p } +(V(\varepsilon  x+x^j)+1)|u |^{p}\right]dx
+\int_{\mathbb{R}^N}F_1 (u )dx-\int_{\mathbb{R}^N}F_2 (u ) dx
$$
which gives that
\begin{equation}\label{9}
\limsup_{\varepsilon\to0^+}I_{\varepsilon}(\hat{u}_\varepsilon^j) =
I_{0}(u) =\mathcal{I}_{0,a}.
\end{equation}
From which, decreasing $\varepsilon^*>0$
if necessary,
$$
I_{\varepsilon}(\hat{u}_\varepsilon^j)<\mathcal{I}_{0,a}+\frac{1}{4}\delta^*,~\forall \varepsilon\in(0,\varepsilon^*),
$$
where $\delta^*>0$ comes from Lemma \ref{PS1}.
Moreover, we easily deduce that $Q_\varepsilon(\hat{u}_\varepsilon^j)\to x^j$
as $\varepsilon\to0^+$ and so $\hat{u}_\varepsilon^j\in \theta_{\varepsilon}^{j}$
by decreasing $\varepsilon^*>0$
if necessary. So, decreasing $\delta^*>0$
if necessary, we have that
$$
\beta_\varepsilon^j<\mathcal{I}_{0,a}+\Upsilon,~\forall \varepsilon\in(0,\varepsilon^*)
$$
which is the first part of the lemma.
To reach the remaining one, if $u\in\partial \theta_{\varepsilon}^{j}$, then
$$
u\in S(a)~\text{and}~|Q_{\varepsilon}(u)-x^{j}|=\rho_{0}>\frac\rho2
$$
leading to $Q_\varepsilon({u})\not\in K_{\frac{\rho_0}{2}}$.
Due to Lemma \ref{PS1}, we find that
$$
I_{\varepsilon}(u)>\mathcal{I}_{0,a}+\frac{\delta^*}{2},~\text{for all}~u\in\partial\theta_\varepsilon^j
~\text{and}~\varepsilon\in(0,\varepsilon^*),
$$
and so
$$
\tilde{\beta}_\varepsilon^j=\inf\limits_{u\in\partial\theta_\varepsilon^j}I_{\varepsilon}(u)\geq \mathcal{I}_{0,a}
+\frac{1}{2}\delta^*,~\forall\varepsilon\in(0,\varepsilon^*),
$$
from where it follows that
$$
\beta_\varepsilon^j<\tilde{\beta}_\varepsilon^j, \text{ for all } \varepsilon\in(0,\varepsilon^*)
$$
finishing the proof of this lemma.
\end{proof}

Now, we are in a position to investigate the existence of multiple critical points for
$I_\varepsilon$ constrained on $S(a)$.

\begin{proposition}\label{existence}
Let $2\leq p<N$ and $a>\tilde{a}>0$. Suppose $\varepsilon\in(0,\varepsilon^*)$ to be fixed, decreasing $\varepsilon^*>0$
if necessary, then $I_\varepsilon|_{S(a)}$ has at least $l$ different nontrivial critical points.
\end{proposition}

\begin{proof}
Given a $j\in\{1,\cdots,l\}$, we could exploit the Ekeland's variational principle to find a sequence $\{u_n^j\}\subset S(a)$ satisfying
$$
I_\varepsilon(u_n^j)\to\beta_\varepsilon^j
$$
and
$$
I_{\varepsilon}(v)-I_{\varepsilon}(u_{n}^{j})\geq-
\frac1n\|v-u_{n}^{j}\|,~\forall v\in\theta_{\varepsilon}^{j}~\mathrm{with}~ v\neq u_{n}^{j}.
$$
It follows from
 Lemma \ref{PS2}
that
$$
\beta_{\varepsilon}^{j}<\tilde{\beta}_{\varepsilon}^{j},~\mathrm{for~all}~\varepsilon\in(0,\varepsilon^*),
$$
and thereby $u_{n}^{j}\in \theta_{\varepsilon}^{j}\backslash\partial\theta_{\varepsilon}^{j}$ for $n$ large enough.
For all $v\in T_{u_n^i}S( a) = \{ w\in X:  \int_{\R^N} |u_n^i|^{p-2}u_n^iwdx= 0\}$,  there
exists a $\zeta>0$ such that the path $\gamma:(-\zeta,\zeta)\to
S( a)$ defined by
$$
\gamma(t)= \frac{a(u_n^j+tv)}{|u_n^j+tv|_p}
$$
which is of class $\mathcal{C}^1((-\zeta,\zeta),S(a))$ and satisfies
$$
\gamma(t)\in\theta_{\varepsilon}^{j}\backslash\partial\theta_{\varepsilon}^{j},~
\forall t\in(-\zeta, \zeta),~\gamma(0)=u_{n}^{j}~\mathrm{and}~\gamma'(0)=v.
$$
Hence,
$$
I_{\varepsilon}(\gamma(t))-I_{\varepsilon}(u_{n}^{j})\geq-\frac{1}{n}\|\gamma(t)-u_{n}^{j}\|,~\forall t\in(-\zeta,\zeta),
$$
and in particular,
$$
\frac {I_{\varepsilon}(\gamma(t))-I_{\varepsilon}(\gamma(0)))}{t}=\frac{I_{\varepsilon}(\gamma(t))
-I_{\varepsilon}(u_{n}^{j})}{t}\geq-\frac{1}{n}\left\|\frac{\gamma(t)-u_{n}^{j}}{t}\right\|=
-\frac{1}{n}\left\|\frac{\gamma(t)-\gamma(0)}{t}\right\|,~\forall t\in(0,\zeta).
$$
Since $I_{\varepsilon}\in \mathcal{C}^1(X, \mathbb{R} ) $, taking the limit of $t\to0^+$, we get

$$
I_{\varepsilon}^{\prime}(u_n^j)v\geq-\frac1n\|v\|.
$$
Now, we replace $v$ with $-v$ to obtain
$$
\sup\{|I_{\varepsilon}^{\prime}(u_{n}^{j})v|:\|v\|\le1\}\le\frac{1}{n},
$$
leading to
$$
I_{\varepsilon}(u_{n})\to\beta_{\varepsilon}^{j}~\mathrm{as}~ n\to+\infty
~\mathrm{and}~\|I_{\varepsilon}|_{S(a)}^{\prime}(u_{n})\|_{X^*}\to0~\mathrm{as}~ n\to+\infty,
$$
 that is, $\{u_n^j\}$ is a $(PS)_{\beta_\varepsilon^j}$ for $I_\varepsilon$ restricted to $S(a)$.
  Since $\beta_\varepsilon^j<\mathcal{I}_{0,a}+\Upsilon$ by Lemma \ref{PS2},
 then Theorem \ref{PScondition} ensures that there is a $u^j$ such that $u_n^j\to u^j$ in $X$. Thus,
$$
 u^j\in \theta_\varepsilon^j, ~I_\varepsilon( u^j) = \beta_\varepsilon^j ~\text{and}~
  I_\varepsilon|_{S( a)} ^{\prime}( u^j) = 0.
$$
Owing to the following facts
$$
Q_{\varepsilon}(u^{i})\in\overline{B_{\rho_{0}}(x^{i})},~Q_{\varepsilon }(u^{j})\in\overline{B_{\rho_{0}}(x^{j})}
$$
and
$$
\overline{B_{\rho_{0}}(x^{i})}\cap\overline{B_{\rho_{0}}(x^{j})}=\emptyset\:\mathrm{for}\:i\neq j,
$$
we conclude that $u^i\neq u^j$ for $i\neq j$ while $1\leq i, j\leq l$. Therefore, $I_{\varepsilon}$ has at least $l$ nontrivial critical points
$(u^j,\lambda^j)$ with $\lambda^j<0$
for all $\varepsilon\in ( 0,\varepsilon^*)$. The proof is completed.
 \end{proof}

In order to study the concentrating behavior of positive solutions
for \eqref{mainequation1}-\eqref{mainequation1a},
we shall depend on the obtained solutions of Problem
\eqref{mainequation2}.
According to Proposition \ref{existence},
for all $2\leq p<N$ and $a>\tilde{a}>0$ and
decreasing $\varepsilon^*>0$
if necessary, there are $l$ couples of $(v^j_\varepsilon,\lambda^j_\varepsilon)\in X\times\R$
such that
$$
v^j_\varepsilon\in \theta_\varepsilon^j,~I_\varepsilon(v^j_\varepsilon)=\beta_\varepsilon^j
~\text{and}~I'_\varepsilon(v^j_\varepsilon)-\lambda_\varepsilon^j\Psi'(v^j_\varepsilon)=0~\text{in}~X^*,
$$
where $j\in\{1,2,\cdots,l\}$, $v^j_\varepsilon(x)>0$ for all $x\in\R^N$ and $\lambda^j<0$.

\begin{lemma}\label{nonvanishing}
Let $2\leq p<N$ and $a>\tilde{a}>0$. Suppose $\varepsilon\in(0,\varepsilon^*)$ to be fixed, decreasing $\varepsilon^*>0$
if necessary, there are $y^j_\varepsilon\in\R^N$, $R_0^j>0$ and $\beta_0^j>0$ such that
$$
\int_{B_{R_0}(y_\varepsilon^j)}|v^j_\varepsilon|^pdx\geq\beta_0^j,
$$
for $j\in\{1,2,\cdots,l\}$. Moreover, the family $\{\varepsilon y_\varepsilon^j\}$ is bounded
and, passing to a subsequence if necessary, $\varepsilon y_\varepsilon^j\to x^j$ as $\varepsilon\to0^+$.
\end{lemma}

 \begin{proof}
If it is not the case, there is a sequence $\{\varepsilon_n\}$ with $\varepsilon_n\to0^+$
such that
$$
\lim_{n\to\infty}\sup_{y\in\R^N}\int_{B_r(y)}|v^j_{\varepsilon_n}|^pdx=0
$$
for all $R>0$. By means of Lion's Vanishing lemma, we would have that $
v^j_{\varepsilon_n}\to0$ in $L^q(\R^N)$ for each $p<q<p^*$ leading to
$F_2(u_n)\to0$ in $L^1(\R^N)$ by property-$(\textbf{P}_2)$.
Owing to $F_1(s)\geq0$ for all $s\in\R$ by property-$(\textbf{P}_1)$, there holds
$\lim\limits_{n\to\infty}I_{\varepsilon_n}(v^j_{\varepsilon_n})\geq0$ which contradicts with the fact that
\begin{equation}\label{11}
 \lim_{n\to\infty}I_{\varepsilon_n}(v^j_{\varepsilon_n})=
\lim_{n\to\infty}\beta_{\varepsilon_n}^i\leq\mathcal{I}_{0,a}+\Upsilon<0.
\end{equation}
So, we can define $\bar{v}^j_\varepsilon(\cdot)=v^j_\varepsilon(\cdot+y^j_\varepsilon)$
and $\{\bar{v}^j_\varepsilon\}$ is bounded with respect to $\varepsilon\in(0,\varepsilon^*)$.
Therefore, there is a $\bar{v}\in X\backslash\{0\}$ such that $\bar{v}_\varepsilon^j\rightharpoonup \bar{v}^j$ in $X$
as $\varepsilon\to0^+$ along a subsequence. Since $\{\bar{v}^j_\varepsilon\}\subset S(a)$
and
$$
I_{\varepsilon }(v^j_{\varepsilon })\geq I_{0}(v^j_{\varepsilon })
=I_{0}(\bar{v}^j_{\varepsilon })\geq \mathcal{I}_{0,a}
$$
jointly with \eqref{9} yields that $\lim\limits_{\varepsilon\to0^+}I_{0}(\bar{v}^j_{\varepsilon })=\mathcal{I}_{0,a}$.
Recalling Theorem \ref{compactness}, we know that $\bar{v}^j_{\varepsilon }\to \bar{v}$
in $X$ as $\varepsilon\to0^+$.
Suppose that $\{\varepsilon y_\varepsilon^j\}$ is unbounded with respect to $\varepsilon\in(0,\varepsilon^*)$
 and so we can assume that there exists a subsequence
$\{\varepsilon_n y_{\varepsilon_n}^j\}$ such that $|\varepsilon_n y_{\varepsilon_n}^j|\to+\infty$
as $n\to\infty$. Exploiting
$\bar{v}^j_{\varepsilon_n}\to \bar{v}$ in $X$,
 \begin{align*}
   I_{\varepsilon_n}(v^j_{\varepsilon_n}) & =
\frac1p\int_{\mathbb{R}^N}\left[|\nabla v^j_{\varepsilon_n}|^{p } +(V(\varepsilon_n x)+1)|v^j_{\varepsilon_n}|^{p}\right]dx
+\int_{\mathbb{R}^N}F_1 (v^j_{\varepsilon_n})dx-\int_{\mathbb{R}^N}F_2 (v^j_{\varepsilon_n}) dx \\
    &=
\frac1p\int_{\mathbb{R}^N}\left[|\nabla \bar{v}^j_{\varepsilon_n}|^{p } +(V(\varepsilon_n x+\varepsilon_n y_n)+1)|\bar{v}^j_{\varepsilon_n}|^{p}\right]dx
+\int_{\mathbb{R}^N}F_1 (\bar{v}^j_{\varepsilon_n})dx-\int_{\mathbb{R}^N}F_2 (\bar{v}^j_{\varepsilon_n}) dx\\
&\to I_{\infty }(\bar{v})
 \end{align*}
 together with \eqref{11} reveals the following inequality
 $$
 \mathcal{I}_{0,a}+\Upsilon\geq I_{\infty }(\bar{v}) \geq \mathcal{I}_{\infty,a }.
 $$
 Due to \eqref{5}, it is impossible by the definition of $\Upsilon$ appearing in
 Theorem \ref{PScondition}. Therefore, up to a subsequence if necessary,
 $\varepsilon y_\varepsilon^j \to x_0^j$ in $\R^N$ as $\varepsilon\to0^+$
 and then the remaining part is to verify that $x_0^j=x^j$.
Actually, we could use a similar argument, or follow the method adopted in
 the case ii)-(2) in the proof of Lemma \ref{PS1}, to conclude that $V(x_0^j)=V_0$.
 Recalling $v_\varepsilon^j\in\theta_\varepsilon^j$ and it would be simple to see that
 $\lim\limits_{n\to\infty}Q_{\varepsilon_n}(v_\varepsilon^j)=x^j_0$,
 one has that $|x^j-x_0^j|\leq \rho_0$. Hence, we must have that $x_0^j=x^j$.
 The proof is completed.
 \end{proof}

\begin{lemma}\label{decay}
Let $2\leq p<N$ and $a>\tilde{a}>0$. Suppose $\varepsilon\in(0,\varepsilon^*)$ to be fixed, decreasing $\varepsilon^*>0$
if necessary, then $v_\varepsilon^j$ possesses a maximum $\eta_\varepsilon^j$
satisfying $V(\varepsilon\eta_\varepsilon^j)\to V(x^j)$ as $\varepsilon\to0^+$ for $j\in\{1,2,\cdots,l\}$.
Moreover, there exist $C_0^j,c_0^j>0$ such that
$$
v_\varepsilon^j(x)\leq C_0^j\exp(-c_0^j|x-\eta_\varepsilon^j|)
$$
for all $\varepsilon\in(0,\varepsilon^*)$ and $x\in \R^N$.
\end{lemma}

 \begin{proof}
Firstly, we analyze some properties of $\bar{v}^j_\varepsilon$.
 Since $\bar{v}^j_\varepsilon(\cdot)=v^j_\varepsilon(\cdot+y^j_\varepsilon)$,
 the definition of $v^j_\varepsilon$ reveals that $(\bar{v}^j_\varepsilon,\lambda^j_\varepsilon)$
 is a couple of weak solution to the problem
 \begin{equation}\label{decay1}
 \left\{
   \begin{array}{ll}
   \displaystyle -  \Delta_p \bar{v}^j_\varepsilon+V(\varepsilon x+ \varepsilon x_\varepsilon^j)
   |\bar{v}^j_\varepsilon|^{p-2}\bar{v}^j_\varepsilon=\lambda_\varepsilon^j |\bar{v}^j_\varepsilon|^{p-2}\bar{v}^j_\varepsilon+
   |\bar{v}^j_\varepsilon|^{p-2}\bar{v}^j_\varepsilon\log|\bar{v}^j_\varepsilon|^p
~\text{in}~\R^N, \\
    \displaystyle     \int_{\R^N}|\bar{v}^j_\varepsilon|^pdx=a^p.
   \end{array}
 \right.
\end{equation}
Recalling the arguments explored in Proposition \ref{existence} and Lemma \ref{nonvanishing},
we derive $\bar{v}^j_\varepsilon\to \bar{v}^j$ in $X$, $\lambda^j_\varepsilon\to \lambda^j$ in $\R$
and $\varepsilon x_\varepsilon^j\to x^j$ in $\R^N$ as $\varepsilon\to0^+$.
So, using \eqref{decay1}, $(\bar{v}^j,\lambda^j)$ is a nontrivial solution to
$$
-  \Delta_p v+V_0
   |v|^{p-2}v=\lambda  |v|^{p-2}v+
   |v|^{p-2}v\log|v|^p
~\text{in}~\R^N.
$$
Similar to Step 1 in the proof of Theorem \ref{Pohozaev},
 there holds $\bar{v}^j\in L^\infty(\R^N)$. Hence, $\bar{v}^j_\varepsilon\in L^\infty(\R^N)$
 and there is a constant $C>0$ independent of $\varepsilon$ such that $|\bar{v}^j_\varepsilon|_\infty\leq C$.
 Indeed, one can further deduce that $\bar{v}^j_\varepsilon\in C_{\mathrm{loc}}^{1,\tau}(\mathbb{R}^N)$ for some $\tau\in(0,1)$.
We postpone the detailed proofs in Lemma A.2 in the Appendix
to give that
 $$
 |\bar{v}_\varepsilon^j|_\infty\geq \rho^j~\text{and}
 ~\lim_{|x|\to+\infty}\bar{v}_\varepsilon^j(x)=0~\text{uniformly in}~\varepsilon\in(0,\varepsilon^*).
 $$
where $\rho^j>0$ is independent of $\varepsilon\in(0,\varepsilon^*)$.

 Secondly, we verify that there exist $\bar{C}_0^j,\bar{c}^j_0>0$ such that
 $\bar{v}_\varepsilon^j(x)\leq \bar{C}_0^j\exp(-\bar{c}^j_0|x|)$ for all
 $\varepsilon\in(0,\varepsilon^*)$ and $x\in\R^N$,
 see Lemma A.3 in the Appendix in detail.

Finally,
let $\varrho_\varepsilon^j$  be a maximum of $\bar{v}_\varepsilon^j$,
 we have that $|\bar{v}_\varepsilon^j(\varrho_\varepsilon^j)|_\infty\geq\rho^j$.
 Since $\lim\limits_{|x|\to\infty}\bar{v}_\varepsilon^j(x)=0$ uniformly in $\varepsilon$,
 there exists a $M_0^j>0$ independent of $\varepsilon$ such that $|\varrho_\varepsilon^j|\leq M_0^j$.
 Recalling $\bar{v}^j_\varepsilon(\cdot)=v^j_\varepsilon(\cdot+y^j_\varepsilon)$, then $y^j_\varepsilon+\varrho_\varepsilon^j$
 is a a maximum of of $v^j_\varepsilon$. Define $\eta^j_\varepsilon=y^j_\varepsilon+\varrho^j_\varepsilon$,
 according to Lemma \ref{nonvanishing} and $|\varrho_\varepsilon^j|\leq M_0^j$,
 we are derived that $\varepsilon\eta^j_\varepsilon\to x^j$
 as $\varepsilon\to0^+$ and hence $V(\varepsilon\eta^j_\varepsilon)\to V(x^j)$
 by the continuity of $V$.
 Moreover, since $\bar{v}_\varepsilon^j(x)\leq \bar{C}_0^j\exp(-\bar{c}^j_0|x|)$ for all $x\in\R^N$
 and $|\varrho_\varepsilon^j|\leq M_0^j$, there holds
 $$
 v ^j_\varepsilon(x)=\bar{v}^j_\varepsilon(x-y^j_\varepsilon)\leq \bar{C}^j_0\exp(-\bar{c}^j_0|x-y^j_\varepsilon|)
 =\bar{C}^j_0\exp(-\bar{c}^j_0|x-\eta^j_\varepsilon+\rho^j_\varepsilon|)\leq C_0^j\exp(-c_0^j|x-\eta_\varepsilon^j|)
 $$
 finishing the proof of this lemma.\end{proof}

 \begin{proof}[\textbf{\emph{Proof of Theorem \ref{maintheorem2}}}]
By Proposition \ref{existence} and Lemma \ref{decay}, we see that Problem \eqref{mainequation2}
admits at least $l$ different couples of solutions $(v_\varepsilon^j,\lambda_\varepsilon^j)\in X\times\R$
with $v_\varepsilon^j(x)>0$ for every $x\in\R^N$ and $\lambda^j_\varepsilon<0$, where $j\in\{1,2,\cdots,l\}$.
Moreover, there exist $C_0^j,c_0^j>0$ such that
$$
v_\varepsilon^j(x)\leq C_0^j\exp(-c_0^j|x-\eta_\varepsilon^j|)
$$
for all $\varepsilon\in(0,\varepsilon^*)$ and $x\in \R^N$.
Let $u_\varepsilon^j(\cdot)=v_\varepsilon^j\left(\cdot/\varepsilon\right)$ and $z_\varepsilon^j=\varepsilon \eta_\varepsilon^j$
for
$j\in\{1,2,\cdots,l\}$, then $(u_\varepsilon^j,\lambda_\varepsilon^j)$ is the
desired solution for
$j\in\{1,2,\cdots,l\}$ and Theorem \ref{maintheorem2} is proved.
 \end{proof}

 \section{The autonomous problem}\label{autonomous}

In this section, we mainly deal with the existence of normalized
 solutions for a class of autonomous $p$-Laplacian
 equations with logarithmic nonlinearities.

 \subsection{The $L^p$-subcritical case}\label{autonomous1}\

 In this subsection,
to study the Problem \eqref{mainequation4}, we need the following minimization problems
 $$
m(a)=\inf_{u\in S(a)}J(u)~\text{and}~
m_r(a)=\inf_{u\in S_r(a)}J(u),
$$
where $S_r(a)=S(a)\cap X_r$ and the variational functional $J$ is defined by \eqref{JJ}.

In order to prove Theorem \ref{maintheorem3}, we are going to introduce the following lemmas.

\begin{lemma}\label{aut1}
Let $2\leq p<N$, then
the functional $J$ is coercive and bounded from below on $S(a)$ for all $a>0$
and there is an constant ${a}_*>0$ such that $m(a)\leq0$ for all $a>{a}_*$.
Moreover, $m(a)=m_r(a)$ for all $a>0$.
\end{lemma}

\begin{proof}
Repeating the calculations in the proofs of Lemmas \ref{realnumber} and \ref{realnumber2},
we can conclude the first part of this lemma and the details are omitted.

Then, we verify that $m(a)=m_r(a)$. Since $S_r(a)\subset S(a)$, one easily sees that $m(a)\leq m_r(a)$.
Thus, we just need to prove that $m(a)\geq m_r(a)$.
Suppose that $\{u_n\}\subset S(a)$ is a minimizing sequence with respect to $m(a)$. Denoting $u_n^*$ to be the Schwarz
symmetric decreasing rearrangement of $u_n$, so the
P\'{o}lya-Szeg\"{o}'s inequality yields that $|\nabla u_n^*|_p\leq |\nabla u_n|_p$.
Noting that $| u_n^*|_r=| u_n|_r$ for every $r\in[p,p^*]$, we obtain that $\{u_n^*\}\subset S_r(a)$.
Since $F_1$ and $F_2$ are nondecreasing in $[0,+\infty]$ by property-$(\textbf{P}_1)$ and
property-$(\textbf{P}_2)$, then the properties of Schwarz rearrangement (see e.g. \cite{LiebLoss})
implies that
$$
 \int_{\mathbb{R}^{N}}F_1(u_{n}^{*})dx=\int_{\mathbb{R}^{N}}F_1(u_{n})dx,~
 \int_{\mathbb{R}^{N}}F_2(u_{n}^{*})dx=\int_{\mathbb{R}^{N}}F_2(u_{n})dx,
$$
 From which, by \eqref{decomposition}, there holds
 $$
\int_{\mathbb{R}^N}|u_n^*|^p\log|u_n^*|^pdx =\int_{\mathbb{R}^N}|u_n|^p\log|u_n|^pdx.
 $$
As a consequence,
 $$
m_{r}(a)=\inf _{u\in S_{r} (a)}J(u)\leq\inf_{u\in S(c)}J(u)=m(a).
$$
The proof is completed.
\end{proof}

\begin{lemma}\label{decreasing}
Let $2\leq p<N$,
then $m(\sqrt[p]{a_1^p+a_2^p})\leq  m(a_1)+m(a_2)$ for all $a_1,a_2>0$.
\end{lemma}

\begin{proof}
In light of the variational functional $J$ is invariant under any translation in $\R^N$, then, adopting the
definition of $m(a)$ and the density of $C_0^\infty(\R^N)$ in $X$,
we deduce that, for every $\epsilon>0$, there
exist two functions $\psi_1,\psi_2\in C_0^\infty(\R^N)$ with
$\supp\psi_1\cap\supp \psi_2= \emptyset$ and $\psi_1\in S(a_1)$,
$\psi_2\in S(a_2)$ such that
 \begin{equation}\label{decreasing1}
 J(\psi_1)\leq m(a_1)+\frac12\epsilon~\text{and}~
 J(\psi_2)\leq m(a_2)+\frac12\epsilon.
 \end{equation}
 Without loss of generality, we assume that
 \begin{equation}\label{decreasing2}
 \operatorname{dist}(\operatorname{supp}\psi_1,\operatorname{supp}\psi_2)\geq n\text{ for some }n\in\mathbb{N}^+.
  \end{equation}
Now define $\psi\triangleq \psi_1+ \psi_2$, since $\psi_1$ and $\psi_2$ have disjoint supports,
then $\psi \in S(\sqrt[p]{a_1^p+a_2^p})$ and
 \begin{equation}\label{decreasing3}
\left\{
  \begin{array}{ll}
 \displaystyle  \int_{\mathbb{R}^{N}} |\nabla  \psi|^{p}dx=
\int_{\mathbb{R}^{N}} |\nabla  \psi_1|^{p}dx
+\int_{\mathbb{R}^{N}} |\nabla  \psi_2|^{p}dx, \\
 \displaystyle   \int_{\mathbb{R}^{N}} |  \psi|^{s}dx=
\int_{\mathbb{R}^{N}} | \psi_1|^{s}dx
+\int_{\mathbb{R}^{N}} |\psi_2|^{s}dx,\quad \forall s\in[p,p^*], \\
 \displaystyle \int_{\mathbb{R}^{N}}|\psi|^p\log|\psi|^pdx
=\int_{\mathbb{R}^{N}}|\psi_1|^p\log|\psi_1|^pdx
+\int_{\mathbb{R}^{N}}|\psi_2|^p\log|\psi_2|^pdx.
  \end{array}
\right.
  \end{equation}
Hence, for $n\in \mathbb{N}^+$ large enough, we are derived from \eqref{decreasing1},
\eqref{decreasing2} and \eqref{decreasing3} that
 $$
 m(\sqrt[p]{a_1^p+a_2^p})\leq J(\psi)  =J(\psi_1)+J(\psi_2)\leq m(a_1)+m(a_2)+\epsilon,
 $$
 and the proof is completed.
 \end{proof}

\begin{lemma}\label{continuous}
Let $2\leq p<N$,
then the mapping $a\mapsto m(a)$ is continuous on $(0,+\infty)$,
where $a_*>0$ comes from Lemma \ref{aut1}.
\end{lemma}

 \begin{proof}
Given an $a>{a}_*$, without loss of generality,
 we let $a_n>{a}_*$ with $a_n\to a$ as $n\to\infty$.
For all $n\in \mathbb{N}$, let $\{u_n\}\subset S(a_n)$ such that
$J(u_n)\leq m(a_n)+\frac1n$. Thanks to Lemma \ref{aut1},
$\{u_n\}$ is uniformly bounded in $X$ and
$$
m(a)\leq J\left({\frac{a}{a_n}}u_n\right)=J(u_n)+o_n(1)\leq m(a_n)+o_n(1).
$$
On the other hand, given a minimizing sequence $\{v_n\}\subset S(a)$ for $m(a)$, it holds that
$$
m(a_n)\leq J\left({\frac{a_n}{a}}v_n\right)=J(v_n)+o_n(1)=m(a)+o_n(1).
$$
The above two facts reveal the desired result and the proof is completed.
 \end{proof}

\begin{lemma}\label{attained}
Let $2\leq p<N$ and $a>a_*$. Assume that $\{u_n\}\subset S_r(a)$ is a minimizing sequence of $m_r(a)$
with $u_n\rightharpoonup u$ in $X_r$ as $n\to\infty$.
If $u\neq0$, then $u_n\to u$ in $X_r$ as $n\to\infty$.
\end{lemma}

\begin{proof}
Obviously, one sees that $u_n\to u$ in $L^s(\R^N)$ for all $s\in(p,p^*)$
and $|u|_p\leq a$ by Fatou's lemma.
Owing to property-$(\textbf{P}_2)$, we have that $F_2(u_n)\to F_2(u_n)$ in $L^1(\R^N)$. To exhibit the proof clearly,
let us divide the proof into two cases.

\textbf{Case 1.} $u_n\to u$ in $L^p(\R^N)$ along a subsequence as $n\to\infty$.\\
In this case, it holds that $u\in S_r(a)$, then
 we immediately have that
$$
\begin{aligned}
m_r(a)&\leq J(u) =\frac{1}{p}\int_{\mathbb{R}^{N}}(|\nabla u|^{p}+|u|^{p})dx+\int_{\mathbb{R}^{N}}
F_1(u)dx-\int_{\mathbb{R}^{N}}
F_2(u)dx -\frac{\mu}{q}\int_{\mathbb{R}^{N}}|u |^{q} dx\\
&\begin{aligned}&\leq\liminf_{n\to\infty}\left(\frac{1}{p}\int_{\mathbb{R}^N}(|\nabla u_n|^p+|u_n|^p)dx+\int_{\mathbb{R}^{N}}
F_1(u_n)dx-\int_{\mathbb{R}^{N}}
F_2(u_n)dx-\frac{\mu}{q}\int_{\mathbb{R}^{N}}|u_n |^{q} dx\right)\end{aligned} \\
&=\liminf_{n\to\infty}J(u_n)=m_r(a)
\end{aligned}
$$
which yields that
$$
\lim_{n\to\infty} \int_{\mathbb{R}^N}(|\nabla u_n|^p+|u_n|^p)dx
=\int_{\mathbb{R}^N}(|\nabla u |^p+|u |^p)dx
$$
and
$$
\lim_{n\to\infty} \int_{\mathbb{R}^{N}}
F_1(u_n)dx=\int_{\mathbb{R}^{N}}
F_1(u )dx.
$$
According to $F_1\in(\Delta_2)$ by Lemma \ref{F1}, then
the above two limits
 provide us that $u_n\to u$ in $X_r$ as $n\to\infty$. The proof is done in this case.

 \textbf{Case 2.} $u_n\not\to u$ in $L^p(\R^N)$ as $n\to\infty$.\\
In this case, denoting $d_n\triangleq|u_n-u|_p$, then, up to a subsequence if necessary,
$$
\lim_{n\to\infty} |u_n-u|_p^p =\lim_{n\to\infty}d_n^p\triangleq d^p>0.
$$
It follows from the Br\'{e}zis-Lieb lemma that $a^p=\lim\limits_{n\to\infty}(d_n^p+|u|_p^p)$. Combining Lemmas \ref{BrezisLieb} and
 \ref{decreasing}-\ref{continuous},
\begin{align*}
 m_r(a) &=m_r\left(\lim\limits_{n\to\infty}\sqrt[p]{d_n^p+|u|_p^p}\right)=\lim\limits_{n\to\infty}m_r\left(\sqrt[p]{d_n^p+|u|_p^p}\right) \\
   &\leq \lim\limits_{n\to\infty}m_r(d_n)+m_r(|u|_p)\leq \lim\limits_{n\to\infty}J(u_n-u)+J(u)\\
   &=\lim\limits_{n\to\infty}J(u_n)= m_r(a).
\end{align*}
Proceeding as the proof in Case 1, we see that $u_n\to u$ in $X_r$ along a subsequence as $n\to\infty$.
The proof of this lemma is completed.
\end{proof}

\begin{proof}[\textbf{\emph{Proof of Theorem \ref{maintheorem3}}}]
First of all, we know that $m(a)\leq0$ for all $a>{a}_*$.
Then, we shall suppose that $\{u_n\}\subset S_r(a)$ is a minimizing sequence
for $m(a)$ by Lemma \ref{aut1}. Exploiting Lemma \ref{aut1} again,
there exists a $u\in X$ such that $u_n\rightharpoonup u$ in $X_r$
along a subsequence.
According to Lemma \ref{attained}, the proof is accomplished if we verify that
$u\neq0$.
Finally, we are ready to deduce that $u\neq0$. Arguing it indirectly, we can assume that $u\equiv0$.
For $m(a)\leq0$, since $F_2(u_n)\to0$ in $L^1(\R^N)$ by property-$(\textbf{P}_2)$
and $u_n\to u$ in $L^s(\R^N)$ for all $s\in(p,p^*)$, then
$$
0\geq\lim_{n\to\infty}J(u_n)=\lim_{n\to\infty}\left(
\frac{1}{p}\int_{\mathbb{R}^{N}}(|\nabla u_n|^{p}+|u_n|^{p})dx+\int_{\mathbb{R}^{N}}
F_1(u_n)dx
\right)\geq \frac1pa^p>0
$$
which is impossible. Therefore, we arrive at the desired result $u\neq0$.
The positivity of $u$ is trivial, we omit it here. The proof is completed.
\end{proof}

\subsection{The $L^p$-supercritical case}\label{autonomous1}\

 In this subsection,
we are going to dispose of the Problem \eqref{mainequation4} with $p+\frac{p^2}{N}<q<p^*$
and $\mu>0$. As explained in Introduction,
we need to study Problem \eqref{mainequation5}
and so we define the minimization problem
$$
m_{\mathcal{R}}(a)=\inf_{u\in S(a)}J_{\mathcal{R}}(u).
$$
The same arguments explored in the proof of Theorem \ref{maintheorem3}
guarantee that there is $\underline{a}^*>0$ independent of $\mathcal{R}$ and $\mu$ such that $m_{\mathcal{R}}(a)\leq0$ for all
$a>\underline{a}^*$.
We would like to point out that it is possible to find such an
$\underline{a}^*>0$ since
$J_{\mathcal{R}}(t\psi)\to-\infty$ as $t\to+\infty$ uniformly in $\mathcal{R}$ and $\mu$.

The next result reveals an important estimate involving the norm in $X$ of the solutions $u_\mathcal{R}$
for the Problem \eqref{mainequation5}.

\begin{lemma} \label{estimate}
Let $2\leq p<N$, $p+\frac{p^2}{N}<q<p^*$ and $\mu>0$.
There exists  $\underline{\mu}^*=\underline{\mu}^*(\mathcal{R})>0$ such that if $\mu\in (0,\underline{\mu}^*)$,
then there is $C>0$
	independent of $\mathcal{R}$ such that the attained function $u_\mathcal{R}$ associated with $m_\mathcal{R}(a)$ satisfies
$|\nabla  u_\mathcal{R}|_p \leq C$ for all $\mathcal{R}>0$.
	
\end{lemma}
\begin{proof}
	Arguing as in Lemma \ref{realnumber}, we see that \eqref{GN} combined with \eqref{ff}
and Property-$(\textbf{P}_2)$ with $\tilde{q}=\bar{q}$ gives
$$J_\mathcal{R}(u)  \geq \frac{1}{p}\int_{\R^N} |\nabla u| ^pdx
- (1+\mu\mathcal{R}^{q-\bar{q}})C_{\bar{q}}\mathbb{C}_{N,p,\bar{q}}a^{\bar{q}(1-\beta_{\bar{q}})}\left(
\int_{\mathbb{R}^{N}} |\nabla u |^{p}dx\right)^{\frac{\bar{q}\beta_{\bar{q}}}{p}},~\forall u \in S(a),
$$
where $\bar{q}\in\left(p,p+\frac{p^2}{N}\right)$.
Fixing $\underline{\mu}^*=\underline{\mu}^*(\mathcal{R})=\frac{1}{\mathcal{R}^{p-\bar{q}}}$, then for all $\mu\in (0,\underline{\mu}^*)$, one gets
$$
  J_\mathcal{R}(u) \geq \frac{1}{p}\int_{\R^N} |\nabla u| ^pdx
 -2 C_{\bar{q}}\mathbb{C}_{N,p,\bar{q}}a^{\bar{q}(1-\beta_{\bar{q}})}\left(
\int_{\mathbb{R}^{N}} |\nabla u |^{p}dx\right)^{\frac{\bar{q}\beta_{\bar{q}}}{p}},~\forall u \in S(a).
$$
Due to $\beta_{\bar{q}} \bar{q} <p$, exploiting the Young's inequality,
there is a constant $C_1>0$ independent of $\mathcal{R}$ such that
	$$
	\left(
\int_{\mathbb{R}^{N}} |\nabla u |^{p}dx\right)^{\frac{\bar{q}\beta_{\bar{q}}}{p}} \leq C_1+ \frac{1}{
4pC_{\bar{q}}\mathbb{C}_{N,p,\bar{q}}a^{\bar{q}(1-\beta_{\bar{q}})}}
\int_{\R^N} |\nabla u| ^pdx, ~\forall u \in X.
	$$
	Hence, there is a constant $C_2>0$ independent of $\mathcal{R}$ such that
	$$
	|\nabla u |_p^{p} \leq 2pJ_\mathcal{R}(u) +C_2, \quad \forall a>\underline{a}^*,~
 \mu\in (0,\underline{\mu}^*), ~\mathcal{R}>0~ \mbox{and} ~u \in S(a).
	$$
Since $J_\mathcal{R}(u_\mathcal{R})=m_\mathcal{R}(a)\leq0$ for all $a>\underline{a}^*$,
we are able to derive the desired result and so the
 proof is completed.
\end{proof}

\begin{proof}[\emph{\textbf{Proof of Theorem \ref{maintheorem4}}}]
By Corollary \ref{corollary}, there are $\underline{a}^*>0$ (independent of $\mathcal{R}$ and
$\mu$) and $\underline{\mu}^*$ such that, for all fixed $a>\underline{a}^*$ and $\mu\in(0,\underline{\mu}^*)$, the couple
 $(u^*_\mathcal{R},\lambda^*_\mathcal{R})\in S_r(a)\times\R$ is a solution of the problem
\begin{align*}
	\left\{
	\begin{aligned}
		&-  \Delta_p u =\lambda |u|^{p-2}u+|u|^{p-2}u\log|u|^p+\mu f_{\mathcal{R}}(u), \hbox{ in } \R^N,\\
		&u(x)>0   \hbox{ in }  \R ^N.\\
	\end{aligned}
	\right.
\end{align*}
Since $\mu\in (0,\underline{\mu}^*)$, the definition of $f_\mathcal{R}$ together with \eqref{ff} leads to
$$
0 \leq \mu f_\mathcal{R}(t) \leq  t^{\bar{q}-1},~\forall t \geq 0 ~\mbox{and}~ \mathcal{R}>0.
$$
As a consequence, $\{u_\mathcal{R}\}$ is bounded in $L^{s}(\R^N)$ for all $\mathcal{R}>0$ and $s \in (p,p^*)$
by Lemma \ref{estimate} and $\{\lambda_\mathcal{R}\}$ is bounded for all $\mathcal{R}>0$. Proceeding as the Step 1 in the proof of
Theorem \ref{Pohozaev}, there is a constant $M\in(0,+\infty)$ that does not depend upon $\mathcal{R}>0$ satisfying
$$
|u_\mathcal{R}|_\infty \leq M, ~\forall \mathcal{R}>0.
$$
Let us fix $\mathcal{R}>M$, then we know that the couple $(u_\mathcal{R}^*,\lambda_\mathcal{R}^*)\in X\times\R$ is weak solution for the
Problem \eqref{mainequation4} if $a>\underline{a}^*$ and $\mu\in (0, \underline{\mu}^*)$. This finishes the proof of Theorem \ref{maintheorem4}.
\end{proof}

\section{Final comments}\label{comments}

Although all of the main results in this article are derived,
as far as we are concerned,
there are some other interesting questions
worth further exploration.

On the one hand, one may naturally wonder that
whether the Problems \eqref{mainequation1}-\eqref{mainequation1a}
admit a ground state solution.
To find the ground state, it suffices to study the existence of ground state solutions
for Problem \eqref{mainequation2}.
We say that $u_0\in X$ is a ground state solution for Problem \eqref{mainequation2}
provided that
\begin{equation*}
I_\varepsilon'(u_0)|_{S(a)}=0~\text{and}~I_\varepsilon(u_0)=\inf\{I_\varepsilon(u):
I_\varepsilon'(u)|_{S(a)}=0~\text{and}~u\in S(a)\},
\end{equation*}
where the variational functional $I_\varepsilon:X\to\R$
defined by \eqref{Ifunctional}. Although we cannot give an affirmative answer to the above question at present,
it will be
  exhibited as a theorem below.

\begin{theorem}\label{comments1}
Let $2\leq p<N$
and $(\hat{V}_1)-(V_2)$.
Suppose additionally that the potential $V$ has no
other strict global minimum points than $\{x^1,x^2,\cdots,x^l\}$.
 Then, there exist ${\hat{a}}^*>0$ and ${\hat{\varepsilon}}^*>0$ such that
 we can  derive at least one ground state solution for \eqref{mainequation1}-\eqref{mainequation1a} among the solutions obtained by Theorem
 \ref{maintheorem2} for all $a>{\hat{a}}^*$ and $\varepsilon\in(0,{\hat{\varepsilon}}^*)$.
\end{theorem}

\begin{remark} If Theorem \ref{comments1} could be proved successfully, one can observe that
the ground state solution
 inherits the properties of concentrating behavior and
exponential decay stated in Theorem \ref{maintheorem2}. Moreover, the positive parameters
${\hat{a}}^*$ is larger than $a^*$ and ${\hat{\varepsilon}}^*$ is smaller than $\varepsilon^*$,
respectively.
\end{remark}

On the other hand,
it would be interesting to handle the existence of solutions
for Problem \eqref{mainequation4} perturbed by
a mass-critical nonlinearity.
More precisely, let us
consider the problem below
\begin{equation}\label{mainequation6}
 \left\{
   \begin{array}{ll}
   \displaystyle -  \Delta_p u =\lambda |u|^{p-2}u+|u|^{p-2}u\log|u|^p+\mu|u|^{\bar{p}-2}u
~\text{in}~\R^N, \\
    \displaystyle     \int_{\R^N}|u|^pdx=a^p.
   \end{array}
 \right.
\end{equation}
where $\mu>0$ is a parameter and $\bar{p}=p+\frac{p^2}{N}$.
It seems difficult to construct a nontrivial solution for Problem \eqref{mainequation6} in our $p$-Laplacian setting so far
because we require the sufficiently small mass $a$ to make sure that the variational functional
is coercive and bounded from below on $S(a)$.
Whereas, we cannot prove that $m(a)\leq0$ in this situation
and it is the essential difference from the classic $p$-Laplacian problems,
namely the logarithmic nonlinearity vanishes in Problem \eqref{mainequation6}.
We conjecture that \eqref{mainequation6} could be solved if one chooses a suitable work
space. Speaking it clearly, motivated by \cite{Cazenave},
there may exist a work space $Z\subset W^{1,p}(\R^N)$
such that $\int_{\R^N}|u|^p|\log |u|^p|dx<\infty$ for all $u\in Z$
and the imbedding $Z\hookrightarrow L^p(\R^N)$ is compact.
If it holds true, one would generalize all the main results in
\cite{SY} to the $p$-Laplacian settings.
What's more, the remained case $1<p<2$ for the results in this paper
would be supplemented.

\appendix
\section{Some technical stuff}

\noindent\textbf{Lemma A.1.} Let $\Omega\subset\R^N$ be an open set,
$\mathcal{L}:\Omega\times\R\times \R^N
\to\R$ is a $C^1$
function and $f\in L^\infty_{\text{loc}}(\Omega)$. If
$\xi\mapsto \mathcal{L}(x,s,\xi)$ is strictly convex for each $(x,s)\in\Omega\times\R$
and $u:\Omega\to\R$ is a locally Lipschitz solution of
\[
-\text{div}\{\nabla_\xi \mathcal{L}(x,u,\nabla u)\}+D_s\mathcal{L}(x,u,\nabla u)=f~\text{in}~\mathcal{D}^\prime(\Omega).
\]
Then
\begin{equation}\label{Appendix1}
\begin{gathered}
\sum_{i,j=1}^N\int_{\Omega}D_ih_j D_{\xi_i}\mathcal{L}(x,u,\nabla u)D_judx
-\int_{\R^N}\big[(\text{div}h)\mathcal{L}(x,u,\nabla u)+h\cdot \nabla_x\mathcal{L}(x,u,\nabla u)\big]dx\hfill\\
\ \ \ \ \ \ \ \  =\int_{\R^N} (h\cdot \nabla u)f(u)dx,~\forall  h\in \mathcal{C}_c^1(\Omega,\R^N). \hfill\\
\end{gathered}
\end{equation}
\vskip3mm
\noindent\textbf{Lemma A.2.}
 Let $(\bar{v}^j_\varepsilon,\lambda^j_\varepsilon)\in X\times\R$
be a couple of weak solution to the Problem
 \eqref{decay1},
 then $$
 |\bar{v}_\varepsilon^j|_\infty\geq \rho^j~\text{and}
 ~\lim_{|x|\to+\infty}\bar{v}_\varepsilon^j(x)=0~\text{uniformly in}~\varepsilon\in(0,\varepsilon^*).
 $$
where $\rho^j>0$ is independent of $\varepsilon\in(0,\varepsilon^*)$.

 \begin{proof}
If the first part is false, we suppose that $|\bar{v}_\varepsilon^j|_\infty\to0$
as $\varepsilon\to0^+$ in the sense of a subsequence. Then, it is simple to
verify that $\bar{v}_\varepsilon^j\to0$ in $X$
which is absurd and thus we just show the second part in detail.
For every $R>0$ and $0<r\leqslant\frac{R}{2}$,
we choose a cutoff function $\eta\in \mathcal{C}_0^{\infty}(\mathbb{R}^{N},[0,1])$
such that $\eta(x)=1$ if $|x|\geqslant R$, and $\eta(x)=0$ if
$|x|\leqslant R-r$ as well as $|\nabla\eta|\leqslant\frac{2}{r}$.
Given $\varepsilon\in(0,\varepsilon^*)$ and $L>1$, define
$$
\left.\bar{v}_{\varepsilon,L}^j(x)
=\left\{\begin{array}{ll}\bar{v}_{\varepsilon}^j(x),&\bar{v}_{\varepsilon}^j(x)\leqslant L,\\
L,&\bar{v}_{\varepsilon}^j(x)\geqslant L,\end{array}\right.\right.
$$
and
$$
\bar{z}^j_{\varepsilon,L}=\eta^{p}(\bar{v}_{\varepsilon,L}^j)^{p(\vartheta-1)}\bar{v}^j_{\varepsilon}~\mathrm{and}~
\bar{w}^j_{\varepsilon,L}=\eta \bar{v}^j_{\varepsilon}(\bar{v}_{\varepsilon,L}^j)^{\vartheta-1}
$$
 with $\vartheta>1$ to be determined later. Taking $\bar{z}_{\varepsilon,L}^j$ as a test function
 in \eqref{decay1}, we obtain
 \begin{align*}
 \int_{\mathbb{R}^{N}}\eta^{p}(\bar{v}_{\varepsilon,L}^j)^{p(\vartheta-1)}|\nabla \bar{v}^j_{\varepsilon}|^{p}dx &
  =-p(\vartheta-1)\int_{\mathbb{R}^{N}}(\bar{v}_{\varepsilon,L}^j)^{p\vartheta-p-1}\eta^{p}\bar{v}^j_{\varepsilon}|\nabla
\bar{v}^j_{\varepsilon}|^{p-2}\nabla \bar{v}^j_{\varepsilon}\nabla \bar{v}_{\varepsilon,L}^j  dx\\
  &\ \ \ \ -p\int_{\mathbb{R}^{N}}\eta^{p-1}(\bar{v}_{\varepsilon,L}^j)^{p(\vartheta-1)}
  \bar{v}_{\varepsilon}^j|\nabla\bar{v}_{\varepsilon}^j|^{p-2}\nabla\bar{v}_{\varepsilon}^j\nabla\eta dx\\
  &\ \ \ \ +\int_{\mathbb{R}^{N}}f(\bar{v}_{\varepsilon}^j)\eta^{p}(\bar{v}_{\varepsilon,L}^j)^{p(\vartheta-1)}\bar{v}^j_{\varepsilon}dx
  -\int_{\mathbb{R}^{N}}
V_\varepsilon(x)|\bar{v}_{\varepsilon}^j|^{p}\eta^{p}
(\bar{v}_{\varepsilon,L}^j)^{p(\vartheta-1)}dx,
\end{align*}
 where $V_\varepsilon(x)=V(\varepsilon x+ \varepsilon x_\varepsilon^j)$ and
 $$
 f(\bar{v}_{\varepsilon}^j)=\lambda_\varepsilon^j |\bar{v}^j_\varepsilon|^{p-2}\bar{v}^j_\varepsilon+
   |\bar{v}^j_\varepsilon|^{p-2}\bar{v}^j_\varepsilon\log|\bar{v}^j_\varepsilon|^p
   =(\lambda_\varepsilon^j-1) |\bar{v}^j_\varepsilon|^{p-2}\bar{v}^j_\varepsilon+F_2^\prime(\bar{v}^j_\varepsilon)-
   F_1^\prime(\bar{v}^j_\varepsilon).
 $$
It follows from property-$(\textbf{P}_1)$ and property-$(\textbf{P}_2)$ with $\tilde{q}\in(p,p^*)$ that
$$
 f(\bar{v}_{\varepsilon}^j)\bar{v}_{\varepsilon}^j
   \leq  (\lambda_\varepsilon^j -1) |\bar{v}^j_\varepsilon|^{p } +pC_{\tilde{q}}|\bar{v}^j_\varepsilon|^{\tilde{q}}
$$
which indicates that
$$
\begin{gathered}
\int_{\mathbb{R}^{N}}\eta^{p}(\bar{v}_{\varepsilon,L}^j)^{p(\vartheta-1)}|\nabla \bar{v}^j_{\varepsilon}|^{p}dx
  \leq
  p\int_{\mathbb{R}^{N}}\eta^{p-1}(\bar{v}_{\varepsilon,L}^j)^{p(\vartheta-1)}
  \bar{v}_{\varepsilon}^j|\nabla\bar{v}_{\varepsilon}^j|^{p-1}|\nabla\eta| dx\hfill\\
 \ \ \ \ +pC_{\tilde{q}}\int_{\mathbb{R}^{N}}\eta^{p}(\bar{v}_{\varepsilon,L}^j)^{p(\vartheta-1)}\bar{v}^j_{\varepsilon}
  |\bar{v}^j_\varepsilon|^{\tilde{q}}dx
  +(\lambda_\varepsilon^j -1-V_0)\int_{\mathbb{R}^{N}}
|\bar{v}_{\varepsilon}^j|^{p}\eta^{p}
(\bar{v}_{\varepsilon,L}^j)^{p(\vartheta-1)}dx. \hfill\\
\end{gathered}
$$
Using the Young's inequality, it holds that
$$\begin{gathered}
 \int_{\mathbb{R}^{N}}\left[\eta^{p}(\bar{v}_{\varepsilon,L}^j)^{p(\vartheta-1)}|\nabla \bar{v}^j_{\varepsilon}|^{p}
-(\lambda_\varepsilon^j -1-V_0)|\bar{v}_{\varepsilon}^j|^{p}\eta^{p}
(\bar{v}_{\varepsilon,L}^j)^{p(\vartheta-1)} \right]dx
\hfill\\
   \ \ \ \
\leq  C_p\int_{\mathbb{R}^{N}}(\bar{v}_{\varepsilon,L}^j)^{p(\vartheta-1)}
 | \bar{v}_{\varepsilon}^j|^p|\nabla\eta|^p dx+pC_pC_{\tilde{q}}\int_{\mathbb{R}^{N}}\eta^{p}
 (\bar{v}_{\varepsilon,L}^j)^{p(\vartheta-1)}
  |\bar{v}^j_\varepsilon|^{\tilde{q}}dx\hfill\\
\end{gathered}
$$
In view of the proof of Theorem \ref{PScondition}, we obtain that $\lambda_\varepsilon^j\leq (\lambda^*)^j<0$
for all $\varepsilon\in(0,\varepsilon^*)$. Exploiting $(\hat{V}_1)$,
there holds $V_0+1\geq0$. Moreover, some simple calculations show that
$$
|\nabla\bar{w}^j_{\varepsilon,L}|^p\leq C_p\vartheta^p\left(
\eta^{p}(\bar{v}_{\varepsilon,L}^j)^{p(\vartheta-1)}|\nabla \bar{v}^j_{\varepsilon}|^{p}
+|\nabla\eta|^{p}(\bar{v}_{\varepsilon,L}^j)^{p(\vartheta-1)}|\bar{v}^j_{\varepsilon}|^{p}\right).
$$
The above facts together with the Sobolev inequality imply that
\begin{align*}
  \left(\int_{\mathbb{R}^{N}}|\bar{w}^j_{\varepsilon,L}|^{p^*}dx\right)^{\frac{p}{p^*}} &
  \leq \tilde{C}_p\vartheta^p\left(
\int_{\mathbb{R}^{N}}(\bar{v}_{\varepsilon,L}^j)^{p(\vartheta-1)}
 | \bar{v}_{\varepsilon}^j|^p|\nabla\eta|^p dx+\int_{\mathbb{R}^{N}}\eta^{p}
 (\bar{v}_{\varepsilon,L}^j)^{p(\vartheta-1)}
  |\bar{v}^j_\varepsilon|^{\tilde{q}}dx
\right)\\
    & \leq
    \tilde{C}_{p,r}\vartheta^p\left(
\int_{R-r\leq|x|\leq R}
 | \bar{v}_{\varepsilon}^j|^{p\vartheta} dx+\int_{|x|\geq R-r}
 (\bar{v}_{\varepsilon}^j)^{p(\vartheta-1)}
  |\bar{v}^j_\varepsilon|^{\tilde{q}}dx
\right).
\end{align*}
Hereafter, we shall fix $t=\sqrt{r}$, $p^*>\frac{pt}{t-1}$
and $\chi=\frac{p^*(t-1)}{pt}>1$. As a consequence,
\begin{align*}
  \left(\int_{\mathbb{R}^{N}}|\bar{w}^j_{\varepsilon,L}|^{p^*}dx\right)^{\frac{p}{p^*}} &
  \leq \tilde{C}_p\vartheta^p\Bigg\{  \left(
\int_{R-r\leq|x|\leq R}
 | \bar{v}_{\varepsilon}^j|^{\frac{p\vartheta t}{t-1}} dx
  \right)^{\frac{t-1}{t}}  \left(  \int_{R-r\leq|x|\leq R}dx \right)^{\frac{1}{t}}\\
  &\ \ \ \ \ \ \ \  \ \  +
  \left(
\int_{|x|\geq R-r}
 | \bar{v}_{\varepsilon}^j|^{\frac{p\vartheta t}{t-1}} dx
  \right)^{\frac{t-1}{t}}  \left(  \int_{|x|\geq R-r}| \bar{v}_{\varepsilon}^j|^{(\tilde{q}-p)t}dx \right)^{\frac{1}{t}}
 \Bigg\}
\end{align*}
Choosing $\tilde{q}=\frac{p(1+t)}{t}$ and applying the Sobolev inequality, we must have that
$$
 \left(\int_{\mathbb{R}^{N}}|\bar{w}^j_{\varepsilon,L}|^{p^*}dx\right)^{\frac{p}{p^*}}
  \leq \tilde{C}_{p,a}\vartheta^p
\left(
\int_{|x|\geq R-r}
 | \bar{v}_{\varepsilon}^j|^{\frac{p\vartheta t}{t-1}} dx
  \right)^{\frac{t-1}{t}}.
$$
From which, since $\bar{w}^j_{\varepsilon,L}=\eta \bar{v}^j_{\varepsilon}(\bar{v}_{\varepsilon,L}^j)^{\vartheta-1}$,
 we can derive that
$$\begin{gathered}
\left(\int_{|x|\geq R}|\bar{v}^j_{\varepsilon,L}|^{p^*\vartheta}dx\right)^{\frac{p}{p^*}}
\leq\left(\int_{|x|\geq R}\eta^{p^*}|\bar{v}^j_{\varepsilon}|^{p^*}
|\bar{v}^j_{\varepsilon,L}|^{p^*(\vartheta-1)}dx\right)^{\frac{p}{p^*}} \hfill\\
\ \ \ \  \leq
\left(\int_{\mathbb{R}^{N}}|\bar{w}^j_{\varepsilon,L}|^{p^*}dx\right)^{\frac{p}{p^*}}
  \leq \tilde{C}_{p,a}\vartheta^p
\left(
\int_{|x|\geq R-r}
 | \bar{v}_{\varepsilon}^j|^{\frac{p\vartheta t}{t-1}} dx
  \right)^{\frac{t-1}{t}}. \hfill\\
\end{gathered}$$
Letting $L\to+\infty$ in the above inequality, there holds
$$
\left(\int_{|x|\geq R}|\bar{v}^j_{\varepsilon}|^{p^*\vartheta}dx\right)^{\frac{p}{p^*}}
\leq \tilde{C}_{p,a}\vartheta^p
\left(
\int_{|x|\geq R-r}
 | \bar{v}_{\varepsilon}^j|^{\frac{p\vartheta t}{t-1}} dx
  \right)^{\frac{t-1}{t}}. $$
Setting $\chi=\frac{p^*(t-1)}{pt}$ and $s=\frac{pt}{t-1}$,
we are derived from the above inequality that
 $$
|\bar{v}^j_{\varepsilon}|_{\chi^{m+1}{s(|x|\geqslant R)}}\leqslant
\tilde{C}_{p,a}^{\sum_{i=1}^{m}\chi^{-i}}\chi^{\sum_{i=1}^{m}i\chi^{-i}}|\bar{v}^j_{\varepsilon}|_{p^{*}(|x|\geqslant R-r)}
$$
and so
$$
 |\bar{v}^j_{\varepsilon} |_{L^\infty(|x|\geqslant R)}\leqslant
 \tilde{C}_{p,a}^{\sum_{i=1}^{m}\chi^{-i}}\chi^{\sum_{i=1}^{m}i\chi^{-i}}|\bar{v}^j_{\varepsilon}|_{p^*(|x|\geqslant R-r)}.
$$
Since $\bar{v}^j_{\varepsilon}\to\bar{v}^j$ in $X$, the last inequality completes the proof.
 \end{proof}

\noindent\textbf{Lemma A.3.}
Let $(\bar{v}^j_\varepsilon,\lambda^j_\varepsilon)\in X\times\R$
be a couple of weak solution to the Problem
 \eqref{decay1}, then there are $C_{0}^j,c_{0}^j>0$ such that
$$
\bar{v}^j_\varepsilon\leq C^j_{0}\exp(-c^j_{0}|x|)
$$
 for all $\varepsilon\in(0,\varepsilon^*)$ and $x\in\R^N$.

\begin{proof}
Since we have derived that $\lambda_\varepsilon^j\to \lambda^j$ as $\varepsilon\to0^+$
and $\bar{v}^j_\varepsilon>0$ for every $x\in\R^N$,
we apply Lemma A.2 to deduce that
$$
\lim_{|x|\to+\infty}\frac{\lambda_\varepsilon^j |\bar{v}^j_\varepsilon|^{p-2}\bar{v}^j_\varepsilon+
   |\bar{v}^j_\varepsilon|^{p-2}\bar{v}^j_\varepsilon\log|\bar{v}^j_\varepsilon|^p}{|\bar{v}^j_\varepsilon|^{p-1}}
   =-\infty~\text{uniformly in}~\varepsilon\in(0,\varepsilon^*).
$$
So, there is an $R>0$ which is independent of $\varepsilon\in(0,\varepsilon^*)$ such that
$$
\lambda_\varepsilon^j |\bar{v}^j_\varepsilon|^{p-2}\bar{v}^j_\varepsilon+
   |\bar{v}^j_\varepsilon|^{p-2}\bar{v}^j_\varepsilon\log|\bar{v}^j_\varepsilon|^p
   \leq \frac{V_0-2}{2}|\bar{v}^j_\varepsilon|^{p-1},~\forall \varepsilon\in(0,\varepsilon^*)~ \text{and}~|x|\geq R.
$$
Denoting the constant $\hat{V}_0=V_0+2\in[1,+\infty)$, then for all $|x|\geq R$, there holds
\begin{align*}
-  \Delta_p \bar{v}^j_\varepsilon+\frac{\hat{V}_0}{2}|\bar{v}^j_\varepsilon|^{p-2}\bar{v}^j_\varepsilon &
=\lambda_\varepsilon^j |\bar{v}^j_\varepsilon|^{p-2}\bar{v}^j_\varepsilon+
   |\bar{v}^j_\varepsilon|^{p-2}\bar{v}^j_\varepsilon\log|\bar{v}^j_\varepsilon|^p-
 \left[V(\varepsilon x+ \varepsilon x_\varepsilon^j)-\frac{V_0+2}{2}\right]
   |\bar{v}^j_\varepsilon|^{p-2}\bar{v}^j_\varepsilon \\
       & \leq \lambda_\varepsilon^j |\bar{v}^j_\varepsilon|^{p-2}\bar{v}^j_\varepsilon+
   |\bar{v}^j_\varepsilon|^{p-2}\bar{v}^j_\varepsilon\log|\bar{v}^j_\varepsilon|^p-
 \frac{V_0-2}{2}
   |\bar{v}^j_\varepsilon|^{p-2}\bar{v}^j_\varepsilon\\
   &\leq0.
   \end{align*}
   Let $\psi^j(x)=C_{0}^j\exp(-c_{0}^j|x|)$
   with $C_{0}^j,c_{0}^j>0$ such that $(c^j_{0})^{p}(p-1)<\frac{\hat{V}}{2}$
   and $\bar{v}^j_\varepsilon(x)\leq C_{0}^j\exp(-c_{0}^jR)$ for all $|x|=R$.
   It follows from some simple calculations that
$$
 -\Delta_{p}\psi^j +\frac{\hat{V} }{2} (\psi^j)^{p-1}
 =(\psi^j)^{p-1}\left(\frac{\hat{V} }{2}-(c^j_{0})^{P}(p-1)+\frac{N-1}{|x|}(c^j_{0})^{p-1}\right) >0,
 ~\text{for all}~|x|\geq R.
$$
Define $\Sigma=\{|x|\geq R\}\cap\{\bar{v}^j_\varepsilon>\psi^j\}$, adopting the following inequality
 $$
 (|x|^{s-2}x-|y|^{s-2}y)\cdot(x-y)\geq0\text{ for all }s>1\text{ and }x,y\in\mathbb{R}^N
 $$
 and choosing $\phi=\max\{\bar{v}^j_\varepsilon-\psi^j,0\}\in W_0^{1,p}(\mathbb{R}^N\backslash B_R)\cap X$
 as a test function in
 $$
  -\Delta_{p}(\bar{v}^j_\varepsilon-\psi^j) +\frac{\hat{V} }{2}\big[(\bar{v}^j_\varepsilon)^{p-1}- (\psi^j)^{p-1} \big]\leq0,
 ~\text{for all}~|x|\geq R
 $$
 to conclude that
$$
0 \geq\int_{\Sigma} (|\nabla \bar{v}^j_\varepsilon|^{p-2}\nabla \bar{v}^j_\varepsilon-|\nabla\psi^j|^{p-2}\psi) \nabla\phi dx
 +\frac{\hat{V} }{2}\int_{\Sigma}\left[(\bar{v}^j_\varepsilon)^{p-1}-(\psi^j)^{p-1}) \right]\phi dx \\
 \geq0.
$$
Therefore, the set $\Sigma\equiv\emptyset$. From which, we know that $\bar{v}^j_\varepsilon\leq\psi^j(x)$ for all $|x|\geq R$ and
$$
\bar{v}^j_\varepsilon\leq\psi^j(x)=C^j_{0}\exp(-c^j_{0}|x|) \text{ for all }|x|\geq R.
$$
Exploiting Lemma A.2 again, $|\bar{v}^j_\varepsilon|_\infty\leq C$ and so
the above inequality holds true for the whole space $\R^N$ by increasing $C_0^j$ to be
large. The proof is completed.
\end{proof}



\end{document}